\documentclass[11pt,a4paper,reqno]{amsart}
\usepackage[english]{babel}
\usepackage[T1]{fontenc}
\usepackage{verbatim}
\usepackage{palatino}
\usepackage{amsmath}
\usepackage{mathabx}
\usepackage{amssymb}
\usepackage{amsthm}
\usepackage{amsfonts}
\usepackage{graphicx}
\usepackage{mathtools}
\usepackage{commath}
\usepackage{overpic}
\usepackage{import}
\usepackage{stackrel}

\usepackage[colorlinks = true, citecolor = black]{hyperref}
\author{Damian D\k{a}browski, Tuomas Orponen}
\address{Department of Mathematics and Statistics\\ University of Jyv\"askyl\"a,
P.O. Box 35 (MaD)\\
FI-40014 University of Jyv\"askyl\"a\\
Finland}
\email{damian.m.dabrowski@jyu.fi}
\email{tuomas.t.orponen@jyu.fi}

\author{Michele Villa}
\address{Department of Mathematics and Statistics\\ University of Jyv\"askyl\"a,
P.O. Box 35 (MaD)\\
FI-40014 University of Jyv\"askyl\"a\\
Finland}
\address{Research Unit of Mathematical Sciences, University of Oulu, P.O. Box 8000, FI-90014, University of Oulu, Finland}
\email{michele.villa@oulu.fi}

\title{Integrability of orthogonal projections, and applications to Furstenberg sets}
\date{\today}
\thanks{All the authors are supported by the Academy of Finland via the projects \emph{Quantitative rectifiability in Euclidean and non-Euclidean spaces} and \emph{Incidences on Fractals}, grant Nos. 309365, 314172, 321896.}

\newcommand{\R}{\mathbb{R}}

\newcommand{\N}{\mathbb{N}}

\newcommand{\calL}{\mathcal{L}}

\newcommand{\calC}{\mathcal{C}}
\newcommand{\calH}{\mathcal{H}}

\newcommand{\calV}{\mathcal{V}}

\newcommand{\spt}{\operatorname{spt}}
\newcommand{\Hd}{\dim_{\mathrm{H}}}

\newcommand{\DD}{\textbf{D}}

\newcommand{\calI}{\mathcal{I}}

\newcommand{\calA}{\mathcal{A}}

\newcommand{\diam}{\operatorname{diam}}
\newcommand{\card}{\operatorname{card}}
\newcommand{\dist}{\operatorname{dist}}

\newcommand{\Rea}{\operatorname{Re}}

\newcommand{\g}{\mathfrak{g}}
\DeclareMathOperator{\im}{im}

\numberwithin{equation}{section}

\theoremstyle{plain}
\newtheorem{thm}[equation]{Theorem}

\newtheorem{lemma}[equation]{Lemma}

\newtheorem{ex}[equation]{Example}
\newtheorem{cor}[equation]{Corollary}
\newtheorem{proposition}[equation]{Proposition}
\newtheorem{question}{Question}

\theoremstyle{definition}

\newtheorem{definition}[equation]{Definition}

\theoremstyle{remark}
\newtheorem{remark}[equation]{Remark}

\begin{document}

\begin{abstract} Let $\mathcal{G}(d,n)$ be the Grassmannian manifold of $n$-dimensional subspaces of $\R^{d}$, and let $\pi_{V} \colon \R^{d} \to V$ be the orthogonal projection. We prove that if $\mu$ is a compactly supported Radon measure on $\R^{d}$ satisfying the $s$-dimensional Frostman condition $\mu(B(x,r)) \leq Cr^{s}$ for all $x \in \R^{d}$ and $r > 0$, then
\begin{displaymath} \int_{\mathcal{G}(d,n)} \|\pi_{V}\mu\|_{L^{p}(V)}^{p} \, d\gamma_{d,n}(V) < \infty, \qquad 1 \leq p < \frac{2d - n - s}{d - s}. \end{displaymath}
The upper bound for $p$ is sharp, at least, for $d - 1 \leq s \leq d$, and every $0 < n < d$.

Our motivation for this question comes from finding improved lower bounds on the Hausdorff dimension of $(s,t)$-Furstenberg sets. For $0 \leq s \leq 1$ and $0 \leq t \leq 2$, a set $K \subset \R^{2}$ is called an $(s,t)$-Furstenberg set if there exists a $t$-dimensional family $\mathcal{L}$ of affine lines in $\R^{2}$ such that $\Hd (K \cap \ell) \geq s$ for all $\ell \in \mathcal{L}$. As a consequence of our projection theorem in $\R^{2}$, we show that every $(s,t)$-Furstenberg set $K \subset \R^{2}$ with $1 < t \leq 2$ satisfies 
\begin{displaymath} \Hd K \geq 2s + (1 - s)(t - 1). \end{displaymath}
This improves on previous bounds for pairs $(s,t)$ with $s > \tfrac{1}{2}$ and $t \geq 1 + \epsilon$ for a small absolute constant $\epsilon > 0$. We also prove a higher dimensional analogue of this estimate for codimension-1 Furstenberg sets in $\R^{d}$. As another corollary of our method, we obtain a $\delta$-discretised sum-product estimate for $(\delta,s)$-sets. Our bound improves on a previous estimate of Chen for every $\tfrac{1}{2} < s < 1$, and also of Guth-Katz-Zahl for $s \geq 0.5151$.
\end{abstract} 

\maketitle

\tableofcontents

\section{Introduction}

This paper is concerned with the $L^{p}$ regularity of orthogonal projections of fractal measures, with applications to $(s,t)$-Furstenberg sets. We introduce the following notation: $\mathcal{M} = \mathcal{M}(\R^{d})$ stands for the space of compactly supported Radon measures on $\R^{d}$, and $\mathcal{M}_{s}$ is the subset of those measures $\mu \in \mathcal{M}$ which satisfy an $s$-dimensional Frostman condition: there exists a constant $C > 0$ such that $\mu(B(x,r)) \leq Cr^{s}$ for all $x \in \R^{d}$ and $r > 0$. The Grassmannian manifold of $n$-dimensional subspaces in $\R^{d}$ is denoted $\mathcal{G}(d,n)$, and the $\mathcal{O}(d)$-invariant probability measure on $\mathcal{G}(d,n)$ is denoted $\gamma_{d,n}$. For $V\in\mathcal{G}(d,n)$, $\pi_V:\R^d\to V$ stands for the orthogonal projection onto $V$. Let us start with the following general question:
\begin{question}\label{mainQ} Let $0 < n < d$, and let $\mu \in \mathcal{M}_{s}$ for some $s > n$. For which values of $1 \leq p,q \leq \infty$ does it hold that
\begin{equation}\label{QEq} \mathfrak{I}(p,q) := \left( \int_{\mathcal{G}(d,n)} \|\pi_{V}\mu\|_{L^{p}(V)}^{q} \, d\gamma_{d,n}(V) \right)^{1/q}  < \infty? \end{equation} 
\end{question}

The question is well-posed, since it is known since the works of Marstrand \cite{Mar}, Kaufman \cite{Ka}, and Mattila \cite{MR0409774} that if $\mu \in \mathcal{M}_{s}$ with $s > n$, then $\pi_{V}\mu \ll \mathcal{H}^{n}|_{V}$ for $\gamma_{d,n}$ almost every plane $V \in \mathcal{G}(d,n)$, and in fact $\mathfrak{I}(2,2) \sim_{d,n} I_{n}(\mu)$, where $I_{t}(\mu)$ stands for the $t$-dimensional Riesz energy of $\mu$. So, at least \eqref{QEq} holds for $p = q = 2$, for every $s > n$. This is not the best one can say: it follows easily from Falconer's Fourier analytic approach \cite{MR673510} and the Sobolev embedding theorem that if $I_{s}(\mu) < \infty$, then $\mathfrak{I}(2n/(2n - s),2) < \infty$, see Section \ref{s:background} for a few more details. Therefore, the answer to Question \ref{mainQ} (where we assume $\mu \in \mathcal{M}_{s}$ instead of $I_{s}(\mu) < \infty$) is positive for all pairs $(p,2)$ with $1 \leq p < 2n/(2n - s)$. For $s > 2n$, the correct interpretation of this is that $\mathfrak{I}(\infty,2) < \infty$. 

The results above only concern pairs of the form $(p,2)$, and the literature seems to be less complete for general pairs $(p,q)$. Of course $\mathfrak{I}(p,q_{1}) \leq \mathfrak{I}(p,q_{2})$ for $q_{1} \leq q_{2}$ by H\"older's inequality, but this observation is unlikely to give any sharp results for $q_{1} \neq q_{2}$. While studying problems related to Furstenberg sets (more on this in Section \ref{s:applications}), we needed to understand pairs of the form $(p,p)$. We show the following:
\begin{thm}\label{mainIntro} Let $\mu \in \mathcal{M}_{s}$ with $s > n$. Then $\mathfrak{I}(p,p) < \infty$ for $1 \leq p < (2d - n - s)/(d - s)$. \end{thm}

The upper bound for "$p$" is sharp for $d \geq 2$, $0 < n < d$, and $d - 1 \leq s \leq d$, as the next example demonstrates. We do not know how sharp Theorem \ref{mainIntro} is for $n < s < d - 1$.  The simplest unknown case occurs for $d = 3, n = 1$, and $1 < s < 2$: what is the supremum of exponents $p \geq 1$ such that $\int_{\mathcal{G}(3,1)} \|\pi_{L}\mu\|_{p}^{p} \, d\gamma_{3,1}(L) < \infty$ for all $\mu \in \mathcal{M}_{s}(\R^{3})$ with $1 < s < 2$? 

\begin{ex}\label{ex1} Fix $d \geq 2$, $0 < n < d$, and $d - 1 \leq s < d$. Let $C \subset L_{0} := \R \times \{\mathbf{0}\} \subset \R^{d}$ be an $(s - (d - 1))$-regular Cantor set (take $C \subset [0,1] \times \{\mathbf{0}\}$ for concreteness), and let $\mu := \nu \times \mathcal{H}^{d - 1}|_{\{\mathbf{0}\} \times B_{d - 1}}$, where $\nu := \mathcal{H}^{s - d + 1}|_{C}$, and $B_{d - 1} \subset \R^{d - 1}$ is the open unit ball. Then $\mu \in \mathcal{M}_{s}$. 

Let $\delta > 0$, and let $\mathcal{G} \subset \mathcal{G}(d,n)$ be the $\delta$-neighbourhood of the submanifold $\mathcal{G}_{0} := \{V \in \mathcal{G}(d,n) : V \supset L_{0}\}$. We record that $\mathcal{G}_{0}$ is a $(d - n)(n - 1)$-dimensional submanifold: the easiest way to get convinced is to note that the restriction "$V \supset L_{0}$" is equivalent to "$V^{\perp} \subset L_{0}^{\perp}$", and the set $\{W \in \mathcal{G}(d,d - n) : W \subset L_{0}^{\perp}\}$ is diffeomorphic to $\mathcal{G}(d - 1,d - n)$, a manifold of dimension $(d - n)((d - 1) - (d - n)) = (d - n)(n - 1)$. Noting that $\gamma_{d,n}$ is an $n(d - n)$-regular measure (see \cite[Proposition 4.1]{MR3044214}), it follows that
\begin{displaymath} \gamma_{d,n}(\mathcal{G}) \sim \delta^{n(d - n)} \cdot \delta^{-\dim \mathcal{G}_{0}} = \delta^{d - n}. \end{displaymath} 
Now, let us consider the projections $\pi_{V}\mu$ for $V \in \mathcal{G}_{0}$, and eventually $V \in \mathcal{G}$. Note first that 
\begin{displaymath} C = \pi_{L_{0}}(\spt \mu) = \pi_{L_{0}}(\pi_{V}(\spt \mu)), \qquad V \in \mathcal{G}_{0}, \end{displaymath}
using that all the planes in $\mathcal{G}_{0}$ contain $L_{0}$. Therefore 
\begin{displaymath} \spt \pi_{V}\mu = \pi_{V}(\spt \mu) \subset B(1) \cap (\pi_{L_{0}}^{-1}(C) \cap V), \qquad V \in \mathcal{G}_{0}. \end{displaymath}
Recalling that $C$ is $(s - d + 1)$-regular, and $L_{0} \subset V$, the set on the right is regular of dimension $(s - d + 1) + (n - 1) = n + s - d$. It can therefore be covered by $\sim \delta^{d - s - n}$ balls in $V$ of radius $\delta$. In particular, $\mathcal{H}^{n}(\spt \pi_{V}\mu) \lesssim \delta^{d - s}$. These arguments were carried for $V \in \mathcal{G}_{0}$, but the conclusion remains valid for $V \in \mathcal{G} = \mathcal{G}_{0}(\delta)$. Now a lower bound for $\|\pi_{V}\mu\|_{L^{p}(V)}$ follows from H\"older's inequality:
\begin{displaymath} \|\pi_{V}\mu\|_{L^{p}(V)}^{p} \gtrsim \mathcal{H}^{n}(\spt \pi_{V}\mu)^{1 - p} \gtrsim \delta^{(d - s)(1 - p)} \qquad V \in \mathcal{G}, \, p \geq 1. \end{displaymath}
Finally,
\begin{displaymath} \int_{\mathcal{G}(d,n)} \|\pi_{V}\mu\|_{L^{p}(V)}^{p} \, d\gamma_{d,n}(V) \gtrsim \gamma_{d,n}(\mathcal{G}) \cdot \delta^{(d - s)(1 - p)} \sim \delta^{d - n + (d - s)(1 - p)}. \end{displaymath}
The right hand side stays bounded as $\delta \to 0$ only if $d - n + (d - s)(1 - p) \geq 0$, or equivalently $p \leq (2d - n - s)/(d - s)$. This matches the upper bound in Theorem \ref{mainIntro}. \end{ex}

\begin{remark} The generalisation of Example \ref{ex1} to the case $s < d - 1$ is not obvious. For $s \geq d - 1$, the measure $\mu$ was defined as Hausdorff measure supported on a union of parallel $(d - 1)$-planes (or pieces thereof, to be accurate). In the case $d = 3$, $n = 1$, and $1 < s < 2$ (for example) it might therefore seem natural to define $\mu := \mathcal{H}^{s}|_{C \times [0,1]}$, where $C \subset \R^{2} \times \{0\}$ has $\mathcal{H}^{s - 1}(C) = 1$. However, with this choice of "$\mu$" it looks like 
\begin{displaymath} \int_{\mathcal{G}(3,1)} \|\pi_{L}\mu\|_{L^{p}(L)}^{p} \, d\gamma_{3,1}(L) < \infty, \qquad 1 \leq p < (3 - s)/(2 - s). \end{displaymath}
This upper bound for "$p$" is higher, for all $s \geq 1$, than the one predicted by Theorem \ref{mainIntro}. \end{remark}

\begin{remark} In addition to the sharpness of Theorem \ref{mainIntro} for $n < s < d - 1$, another special case of Question \ref{mainQ} is worth highlighting: for $\mu \in \mathcal{M}_{s}(\R^{2})$ with $s > 1$, determine the supremum of exponents $p \geq 1$ such that $\mathfrak{I}(p,1) < \infty$. This is closely related to the question Peres and Schlag raise in \cite[\S9.2(ii)]{MR1749437}. More precisely, they ask for the value of $p(s) := \sup \{p \geq 1 : \pi_{L}\mu \in L^{p} \text{ for a.e. } L \in \mathcal{G}(2,1), \text{ for all } \mu \in \mathcal{M}_{s}(\R^{2})\}$. We do not even have a good guess for the right answer. Measures supported on concentric unions of circles give one upper bound for $p(s)$, and measures supported on Furstenberg sets give another one. These upper bounds do not coincide. \end{remark}

\begin{remark} While the problem regarding $\mathfrak{I}(p,1)$ seems difficult, and most likely unsolved, Theorem \ref{mainIntro} may be known to experts in harmonic analysis: it is essentially an $L^{p} \to L^{p,\alpha}$ estimate for the $(d - n)$-plane transform, and there is a formidable amount of literature on estimating this operator. For the pairs $(d,n) = (d,1)$, $d \geq 2$, one could, with a little effort, deduce Theorem \ref{mainIntro} from the work of Littman \cite{MR155146}, by first expressing the $(d - 1)$-plane transform (also known as the Radon transform) as an averaging operator over the $(d - 1)$-dimensional paraboloid in $\R^{d}$, see the identities (2.1) and (2.9) in Christ's paper \cite{MR3329848}, and eventually exploiting the curvature of the paraboloid, as Littman does. 

For more general dimensions and co-dimensions, Strichartz \cite[Theorem 2.2]{MR782573} proves $L^{p} \to L^{p,\alpha}$ estimates for the $n$-plane transform in $\R^{d}$, but only for $1 < p \leq 2$ (there is a good reason, see Remark \ref{r:strichartz}). Theorem \ref{mainIntro} is also closely related to the papers of Drury \cite{MR684547}, D. Oberlin and Stein \cite{MR667786}, and D. Oberlin \cite{MR2320409,MR2994685}. In these works, the authors prove sharp $L^{p}$ to $L^{q}$ estimates for the Radon transform, but as far as we can see, they do not contain the $L^{p}$ to $L^{p}$-Sobolev result we need for our purposes. Mixed norm estimates for Radon transforms are intimately connected with Kakeya and Besicovitch $(n,k)$-set problems, and there is a wealth of literature for $d \geq 3$, see for example \cite{MR1040963,MR1891202,MR2719423,MR2213456,MR1681585}. Smoothness and integrability estimates for Radon transforms are also of interest to mathematicians working on inverse problems: see the book \cite{MR856916} by Natterer, and in particular the bibliographical notes at the end of Section II. In summary, there is a non-zero probability that Theorem \ref{mainIntro} is covered by existing literature, but we could not easily find it, and in any case our proof is self-contained and fairly elementary. \end{remark}

\subsection{Applications}\label{s:applications} We then move to the applications which motivate Question \ref{mainQ} for the pairs $(p,p)$. The main one concerns \emph{Furstenberg $(s,t)$-sets}, defined as follows. A set $K \subset \R^{2}$ is called an $(s,t)$-Furstenberg set if there exists a family $\mathcal{L}$ of affine lines with $\Hd \mathcal{L} = t$ such that $\Hd (K \cap \ell) \geq s$ for all $\ell \in \mathcal{L}$. Here the dimension "$\Hd \mathcal{L}$" is defined by viewing $\mathcal{L}$ as a subset of the metric space $(\mathcal{A}(2,1),d_{\mathcal{A}})$, the \emph{affine Grassmannian} of all lines in the plane. We postpone the precise definition of the metric $d_{\mathcal{A}}$ to Section \ref{s:prelim}, see \eqref{rev4}.

The case $t = 1$ has attracted the most attention: Wolff \cite{Wolff99} introduced the problem in the late 90s and showed that every $(s,1)$-Furstenberg set $K \subset \R^{2}$, $0 < s \leq 1$, satisfies 
\begin{equation}\label{wolff} \Hd K \geq \max\{2s,\tfrac{1}{2} + s\}. \end{equation}
Wolff also conjectured that the sharp estimate should be $\Hd K \geq \tfrac{1}{2} + \tfrac{3s}{2}$. In part relying on the work of Katz and Tao \cite{MR1856956}, Bourgain in 2003 managed to improve on Wolff's estimate by an "$\epsilon$" in the case $s = \tfrac{1}{2}$. For $\tfrac{1}{2} < s < 1$, a similar $\epsilon$-improvement was achieved in 2021 by the second author and Shmerkin \cite{2021arXiv210603338O}, partly relying on the earlier paper \cite{Orponen20}. In fact, \cite{2021arXiv210603338O} established that $\Hd K \geq 2s + \epsilon(s,t)$ for Furstenberg $(s,t)$-sets with $0 < s < 1$ and $t \in (s,2]$. For $0 < s \leq \tfrac{1}{2} - \epsilon$, Wolff's estimate remains the strongest one, although an $\epsilon$-improvement for the \emph{packing dimension} of $s$-Furstenberg sets in this region of parameters was obtained by Shmerkin \cite{2020arXiv200615569S} in 2020.

For more general $t \in [0,2]$, lower bounds for Furstenberg $(s,t)$-sets have been recently obtained by Molter and Rela \cite{MR2910763}, H\'era \cite{MR3973547}, H\'era, M\'ath\'e, and Keleti \cite{MR4002667}, Lutz and Stull \cite{MR4179019}, and H\'era, Shmerkin, and Yavicoli \cite{2020arXiv200111304H}. The best previous bounds for the number
\begin{displaymath} \gamma(s,t) := \inf\{\Hd K : K \subset \R^{2} \text{ is an $(s,t)$-Furstenberg set}\} \end{displaymath}
are the following (combining contributions from all the papers cited above):
\begin{displaymath} \gamma(s,t) \geq \begin{cases} s +t & \text{for } s \in (0,1] \text{ and } t \in [0,s], \\ 2s + \epsilon(s,t) & \text{for } s \in (0,1] \text{ and } t \in (s,2s], \\ s + \tfrac{t}{2} & \text{for } s \in (0,1] \text{ and } t \in (2s,2]. \end{cases} \end{displaymath} 
Our new result concerns the "high dimensional" region where $s > \tfrac{1}{2}$ and $t > 1$:
\begin{thm}\label{mainFurstenberg} Let $0 < s \leq 1$ and $1 < t \leq 2$. Then every $(s,t)$-Furstenberg set $K \subset \R^{2}$ satisfies 
\begin{equation}\label{eq:Furstenberg} \Hd K \geq 2s + (1 - s)(t - 1). \end{equation}
More generally, every $(d - 1,s,t)$-Furstenberg set $K \subset \R^{d}$, with $d \geq 2$, $1 < t \leq d$ and $0 < s \leq d - 1$ satisfies
\begin{equation}\label{eq:FurstenbergGeneral} \Hd K \geq (2s + 2 - d) + \frac{(t - 1)(d - 1 - s)}{d - 1}. \end{equation}
\end{thm}
We postpone the definition of $(d - 1,s,t)$-Furstenberg sets for a moment, see Section \ref{s:higherFurstenberg}. The estimate \eqref{eq:Furstenberg} is stronger than the bound $s + t/2$, due to H\'era \cite{MR3973547}, in the range $s > \tfrac{1}{2}$ and $t > 1$, and also improves on the bound $2s + \epsilon(s,t)$ for $(1 - s)(t - 1) > \epsilon(s,t)$ (the constant $\epsilon(s,t) > 0$ is very small). We derive Theorem \ref{mainFurstenberg} as a corollary of a following $\delta$-discretised incidence result, which also gives some information in higher dimensions. To state the result, we first define the notion of $(\delta,s,C)$-sets:
\begin{definition}[$(\delta,s,C)$-set]\label{def:deltaset} Let $0 \leq s <\infty$, $0<\delta<1$, and $C > 0$. Given a metric space $(X,d)$, a bounded set $P \subset X$ is called a $(\delta,s,C)$-set if for every $\delta \leq r \leq 1$ and every ball $B\subset X$ of radius $r$ we have
\begin{displaymath}\ 
	|P \cap B|_{\delta} \leq C \cdot |P|_{\delta} \cdot r^{s}. 
\end{displaymath}
Here $|A|_{\delta}$ denotes the $\delta$-covering number of $A$, i.e. the minimal number of $\delta$-balls needed to cover $A$ (we set $|A|_{\delta} := \infty$ if $A$ cannot be covered by finitely many $\delta$-balls). \end{definition}

In the following, if $A \subset \R^{d}$, and $r > 0$, then $A(r) := \{x \in \R^{d} : \dist(x,A) \leq r\}$.

\begin{thm}\label{t:incidences} Let $0 < n < d$ and $C, C_F \geq 1$. Let $\mathcal{V} \subset \mathcal{A}(d,n)$ be a $\delta$-separated set of $n$-planes, and let $P \subset B(1) \subset \R^{d}$ be a $\delta$-separated $(\delta,t,C_F)$-set with $t > d - n$. For $r>0$ let $\mathcal{I}_{r}(P,\mathcal{V}) = \{(p,V) \in P \times \mathcal{V} : p \in V(r)\}$. Then, for every $\varepsilon>0$ we have
	\begin{equation*}
		|\mathcal{I}_{C\delta}(P,\mathcal{V})| \lesssim_{C,d,\varepsilon,t} \delta^{-\varepsilon} \cdot C_{F} \cdot |P| \cdot |\mathcal{V}|^{n/(d+n-t)} \cdot \delta^{n(t+1-d)(d-n)/(d+n-t)}.
	\end{equation*}
	\end{thm}

To derive Theorem \ref{mainFurstenberg} from Theorem \ref{t:incidences}, the incidence result needs to be applied to the dual set of (a suitable discretisation of) "$\mathcal{L}$", the $t$-dimensional set of lines appearing in the definition of $(s,t)$-Furstenberg sets. While it is unlikely that Theorem \ref{mainFurstenberg} is sharp for any $s \in (0,1)$ or $t \in [1,2)$, Theorem \ref{t:incidences} is fairly sharp in the plane, essentially because the set $\mathcal{V}$ is "only" assumed to be $\delta$-separated. This matter is discussed further in Section \ref{s:sharpness}, see Proposition \ref{counterProp} and Remark \ref{rem1}.

Theorem \ref{t:incidences}, or rather its dual version, also allows us to make progress on the \emph{$\delta$-discretised sum-product problem} in the "supercritical" range $t > \tfrac{1}{2}$:
\begin{cor}\label{cor:sumProduct} Let $\delta \in (0,1]$, $s,t,t' \in [0,1]$ with $t+t'>1$, and $c,c' > 0$. Let $A,B,C \subset [1,2]$ be $\delta$-separated sets such that $|A| = \delta^{-s}$, $B$ is a $(\delta,t,c)$-set and $C$ is a $(\delta,t',c')$-set. Then, 
\begin{displaymath} \max\{|A + B|_{\delta},|A \cdot C|_{\delta}\} \gtrsim_{\alpha,s,t,t',c,c'} \delta^{- \alpha}|A|, \qquad \alpha < \tfrac{(t + t' - 1)(1 - s)}{2}. \end{displaymath}
\end{cor}
We are grateful to Josh Zahl for telling us that Corollary \ref{cor:sumProduct} follows from Theorem \ref{t:incidences} combined with an argument of Elekes \cite{MR1472816}, see Section \ref{s:sumProduct} for the details. Corollary \ref{cor:sumProduct} applied with $A=B=C$ (and assuming that $A$ is a $(\delta,t)$-set with $t\in (\tfrac{1}{2},1)$) improves on recent results of Chen \cite{MR4066556} for every $t \in(\tfrac{1}{2},1)$, and of Guth, Katz, and Zahl \cite{MR4283564} for $1 > t > (\sqrt{1113} - 21)/24 \approx 0.5151$. We refer the reader to these papers for more background and references on the $\delta$-discretised sum-product problem. Since $(2t - 1)(1 - s)/2 > 0$ for $t \in (\tfrac{1}{2},1)$ and $s \in (0,1)$, if we assume that $B=C$ and $B$ is a $(\delta,t)$-set with $t \in (\tfrac{1}{2},1)$, Corollary \ref{cor:sumProduct} also implies that $\max\{|A + B|_{\delta},|A \cdot B|_{\delta}\} \gg |B|$ in cases where $A$ is substantially smaller than $B$ (to be precise, this works when $s > 1/(3 - 2t)$; note that $1/(3 - 2t) < t$ for $t \in (\tfrac{1}{2},1)$, so the range $s \in (1/(3 - 2t),t)$ is non-empty). 

\subsubsection{Higher dimensional Furstenberg sets}\label{s:higherFurstenberg} Theorem \ref{mainFurstenberg} mentions the notion of $(n,s,t)$-Furstenberg sets in $\R^{d}$. These are defined just like $(s,t)$-Furstenberg sets, except that the set $\mathcal{L} \subset \mathcal{A}(2,1)$ is replaced by a $t$-dimensional set $\mathcal{V} \subset \mathcal{A}(d,n)$ of affine $n$-planes. Thus, a set $K \subset \R^{d}$ is called an $(n,s,t)$-Furstenberg set if there exists a family $\mathcal{V} \subset \mathcal{A}(d,n)$ with $\Hd \mathcal{V} = t$ such that $\Hd (K \cap V) \geq s$ for all $V \in \mathcal{V}$. The dimension "$\Hd \mathcal{V}$" is defined relative to the metric on $\mathcal{A}(d,n)$, see Section \ref{s:prelim}. Since Theorem \ref{mainFurstenberg} is deduced via duality from Theorem \ref{t:incidences}, we only obtain information about the case $n = d - 1$. 

Furstenberg $(n,s,t)$-sets have been studied in many of the papers cited above, see \cite{MR3973547,MR4002667,2020arXiv200111304H}. Additionally, finite field versions of $(n,s,t)$-Furstenberg sets in $\mathbb{F}_{p}^{d}$ have been considered by Ellenberg and Erman \cite{MR3554237}, Dhar, Dvir, and Lund \cite{2019arXiv190903180D}, and Zhang \cite{MR3333966}. We also mention the paper of Zhang \cite{MR3595894}, where the author studies a discrete variant of the Furstenberg set problems in $\R^{d}$. 

We only discuss the existing bounds in the case $n = d - 1$. H\'era in \cite{MR3973547} proves that every $(d - 1,s,t)$-Furstenberg set $K \subset \R^{d}$ with $(s,t) \in (d - 2,d - 1] \times (0,d]$ satisfies $\Hd K \geq s + t/d$. In \cite{MR4002667}, H\'era, M\'ath\'e, and Keleti prove the lower bound $\Hd K \geq 2s - d + 1 + \min\{t,1\}$ for all $(s,t) \in (0,d - 1] \times (0,d]$. Clearly \eqref{eq:FurstenbergGeneral} improves on the H-K-M bound for all $t \in (1,d]$. One may calculate that \eqref{eq:FurstenbergGeneral} also improves on H\'era's bound for $(s,t) \in (d - 2 + \tfrac{1}{d},d - 1] \times (1,d]$.

\subsection{Outline of the paper} The proof of Theorem \ref{mainIntro} is conceptually quite straightforward: it is based on complex interpolation between the cases $s = n$ and $s = d$. This argument is heavily influenced by the paper \cite{MR782573} of Strichartz. The technical details nevertheless take some work, see Section \ref{s:Lp}. Section \ref{s:prelim} only contains some preliminaries.

Theorem \ref{mainFurstenberg} on $(d - 1,s,t)$-Furstenberg sets is reduced to the incidence estimate in Theorem \ref{t:incidences} by applying point-plane duality, and standard discretisation arguments. The details are contained in Section \ref{s:furstenberg}. The proof of Theorem \ref{t:incidences} is carried out in Section \ref{s:incidence}. The idea is easiest to explain in the plane. Imagine that $P \subset \R^{2}$ is a $\delta$-separated $(\delta,t)$-set (see Definition \ref{def:deltaset}) with $1 < t \leq 2$, and let $\mathcal{L} \subset \mathcal{A}(2,1)$ be a $\delta$-separated line family with excessively many $\delta$-incidences with $P$. Let $\mu \in \mathcal{M}_{t}$ with $\spt \mu = P(\delta)$. If the word "excessive" is interpreted as the serious failure of Theorem \ref{t:incidences}, then it turns out that many \emph{radial projections} $\rho_{x}\mu$ of $\mu$ relative to base points $x \in \spt \mu = P(\delta)$ are singular. (The reader should be warned that $\rho_{x}\mu$ is not precisely the push-forward of $\mu$ under $y \mapsto \rho_{x}(y)$, see \eqref{e:mux} for the proper definition.)

This sounds like a contradiction: a result of the second author \cite{MR3892404} says that the radial projections of a $t$-dimensional measure, $t > 1$, relative to its own base points are (typically) absolutely continuous with a density in $L^{p}$, for some $p > 1$. The result in \cite{MR3892404} is proved via relating the radial and orthogonal projections of $\mu$ by the following formula:
\begin{displaymath} \int \|\rho_{x}\mu\|_{L^{p}(S^{1})}^{p} \, d\mu(x) = \int_{S^{1}} \|\pi_{e}\mu\|_{L^{p + 1}(\R)}^{p + 1} \, d\mathcal{H}^{1}(e). \end{displaymath}
For a higher dimensional generalisation, see \eqref{form4}. With this identity in hand, we may estimate the right hand side by appealing to Theorem \ref{mainIntro}: it is finite for all $p + 1 < (3 - t)/(2 - t)$, or equivalently $p < 1/(2 - t)$. Pitting this information against the hypothetical singularity of the radial projections $\rho_{x}\mu$ yields Theorem \ref{t:incidences}. A similar approach also works in higher dimensions and co-dimensions: the details can be found in Section \ref{s:incidence}.

As we already mentioned above, Section \ref{s:sharpness} contains a family of examples indicating the sharpness of Theorem \ref{t:incidences}. These examples will also indicate where the numerology in the lower bound \eqref{eq:Furstenberg} comes from.

\subsection{Acknowledgements} As already mentioned below Corollary \ref{cor:sumProduct}, we are grateful to Josh Zahl for pointing out how to derive it from Theorem \ref{t:incidences}. We are also grateful to the anonymous reviewers for reading a draft of the paper carefully, and giving plenty of useful feedback to improve our exposition.

\section{Preliminaries}\label{s:prelim}

We will write $f\lesssim g$ as an abbreviation for the inequality $f\le C g$, where $C>0$ is an absolute constant. If the constant $C$ depends on a parameter $a$, we will write $f\lesssim_a g$. Furthermore, $f\sim g$ and $f\sim_a g$ will denote $g\lesssim f\lesssim g$ and $g\lesssim_a f\lesssim_a g$, respectively.

In addition to the notations "$f \lesssim g$" and "$f \sim g$", we will also employ "$f \lessapprox g$" and "$f \approx g$". The notation $f \lessapprox g$ refers to an inequality of the form $f \leq C \cdot (\log(1/\delta))^{C} \cdot g$, where $C > 0$ is an absolute constant, and $\delta > 0$ is a parameter (always a "scale") which will be clear from context. The two-sided inequality $f \lessapprox g \lessapprox f$ is abbreviated to $f \approx g$.

The notation $B(x,r)$ stands for the closed ball of radius $r>0$ around $x$. Usually $x\in\R^d$, in which case $B(x,r)$ denotes the usual Euclidean ball. Occasionally, $x$ will belong to another metric space (e.g., the Grassmannian $\mathcal{G}(d,n),$ or the circle $S^1$). In such cases $B(x,r)$ denotes the metric ball. Sometimes we will write $B(r)$ instead of $B(0,r)$.

Our main result on incidences, Theorem \ref{t:incidences}, was been formulated in terms of \emph{$(\delta,s,C)$-sets}. We recall (from Definition \ref{def:deltaset}) that a bounded set $P \subset X$ in a metric space $(X,d)$ is called a $(\delta,s,C)$-set if
\begin{equation}\label{eq:deltaset} 
	|P \cap B(x,r)|_{\delta} \leq C \cdot |P|_{\delta} \cdot r^{s}, \qquad x \in X, \, \delta \leq r \leq 1. 
\end{equation}
If the value of the constant $C > 0$ is irrelevant, we may also talk casually about $(\delta,s)$-sets. For more information about basic properties of $(\delta,s)$-sets, see \cite[Section 2.1]{2021arXiv210603338O}. Our notion of $(\delta,s)$-sets is not entirely canonical: an alternative common definition is where \eqref{eq:deltaset} is replaced by $|P \cap B(x,r)|_{\delta} \leq (r/\delta)^{s}$. The definitions coincide when $|P|_{\delta} \sim \delta^{-s}$. One difference between the definitions is worth noting: our definition implies that if $P$ is a non-empty $(\delta,s,C)$-set, then $|P|_{\delta} \geq \delta^{-s}/C$. This follows from \eqref{eq:deltaset} applied to any ball $B(x,\delta)$ with $x \in P$. In contrast, the alternative definition $|P \cap B(x,r)|_{\delta} \leq (r/\delta)^{s}$ rather implies an upper bound $|P|_{\delta} \leq \delta^{-s}$, at least if $\diam(P) \leq 1$. 

In the paper we will only consider $(\delta,s)$-sets in the Euclidean space $(\R^d,|\cdot|)$, and in the affine Grassmannian $(\calA(d,n),d_{\calA})$. The metric $d_{\calA}$ is defined as in \cite[\S 3.16]{zbMATH01249699}: given $V,W\in\calA(d,n)$, let $V_0,W_0\in\mathcal{G}(d,n)$ and $a\in V_0^{\perp},\ b\in W_0^\perp,$ be the unique $n$-planes and vectors such that $V=V_0+a$ and $W=W_0+b$. The distance between $V$ and $W$ is given by
\begin{equation}\label{rev4}
d_{\calA}(V,W) := \norm{\pi_{V_0}-\pi_{W_0}}_{op} + |a-b|,
\end{equation}
where $\lVert{\cdot}\rVert_{op}$ denotes the operator norm. Note that $\mathcal{G}(d,n)$ can be seen as a submanifold of $\mathcal{A}(d,n)$, and the restriction of $d_{\calA}$ to $\mathcal{G}(d,n)\times \mathcal{G}(d,n)$ defines a metric on $\mathcal{G}(d,n)$.

For a set $A\subset\R^d$ and $\delta>0$, $A(\delta)$ will denote the $\delta$-neighbourhood of $A$. 

\section{$L^{p}$-regularity of projections}\label{s:Lp}

\subsection{Background}\label{s:background} Let $0 < n < d$, let $\mathcal{G}(d,n)$ be the Grassmannian of $n$-dimensional subspaces of $\R^{d}$, and let $\mathcal{M} = \mathcal{M}(\R^{d})$ be the family of compactly supported Radon measures on $\R^{d}$. In this section we investigate the $L^{p}$-regularity of the projections of $s$-dimensional Frostman measures $\mu \in \mathcal{M}$ to planes $V \in \mathcal{G}(d,n)$. 

It is classical that if $s > n$, and $\mu \in \mathcal{M}$ satisfies the $s$-dimensional Frostman condition $\mu(B(x,r)) \lesssim r^{s}$ for balls $B(x,r) \subset \R^{d}$, then 
\begin{displaymath} \int_{\mathcal{G}(d,n)} \|\pi_{V}\mu\|_{2}^{2} \, d\gamma_{d,n}(V) < \infty, \end{displaymath}
where $\gamma_{d,n}$ is the $\mathcal{O}(d)$-invariant probability measure on $\mathcal{G}(d,n)$. This can be easily deduced from the potential theoretic method due to Kaufman \cite{Ka} in $\R^{2}$ and Mattila \cite{MR0409774} in higher dimensions, or see \cite[Theorem 9.7]{zbMATH01249699} for a textbook reference. In fact, a little more is known: if the $s$-dimensional Riesz energy $I_{s}(\mu)$ is finite, $s \geq n$ (in particular: if $\mu(B(x,r)) \lesssim r^{t}$ for some $t > s$), then $\gamma_{d,n}$ almost every projection $\pi_{V}\mu$ lies in the fractional Sobolev space $H^{(s - n)/2}(V) \cong H^{(s - n)/2}(\R^{n})$, and
\begin{equation}\label{form29} \int_{\mathcal{G}(d,n)} \int_{V} |\widehat{\pi_{V}\mu}(\xi)|^{2}|\xi|^{s - n} \, d\mathcal{H}^{n}(\xi) \, d\gamma_{d,n}(V) \lesssim I_{s}(\mu). \end{equation}
This approach via Fourier transforms was pioneered by Falconer \cite{MR673510}, and the estimate \eqref{form29} can be found for example in \cite[(5.14)]{MR3617376}. By the Sobolev embedding theorem \cite[Theorem 6.5]{MR2944369}, it follows for that $\pi_{V}\mu$ has a density in $L^{p^{\star}}$ for $\gamma_{d,n}$ a.e. $V \in \mathcal{G}(d,n)$, with $p^{\star} := p^{\star}(n,s) := 2n/(2n - s)$, and indeed
\begin{equation}\label{form28} \int_{\mathcal{G}(d,n)} \|\pi_{V}\mu\|_{L^{p^{\star}(n,s)}(V)}^{2} \, d\gamma_{d,n}(V) \lesssim I_{s}(\mu), \qquad n \leq s < 2n. \end{equation}
For $2n < s < d$, one can even deduce that $\pi_{V}\mu \in C_{c}(V)$ for $\gamma_{d,n}$ a.e. $V \in \mathcal{G}(d,n)$, and $V \mapsto \|\pi_{V}\mu\|_{L^{\infty}(V)} \in L^{2}(\mathcal{G}(d,n))$, see the proof of \cite[Theorem 5.4(c)]{MR3617376} applied to $\pi_{V}\mu$.

\subsection{New results} We do not know how sharp the facts from Section \ref{s:background} are under the hypothesis $I_{s}(\mu) < \infty$, but they are certainly unsatisfactory under the $s$-Frostman assumption $\mu(B(x,r)) \lesssim r^{s}$. To see this, consider the situation in $\R^{2}$. If $\mu \in \mathcal{M}(\R^{2})$ with $\mu(B(x,r)) \lesssim r^{t}$ for some $1 < t \leq 2$, then one may deduce from the "mixed norm estimate" \eqref{form28} that $L \mapsto \|\pi_{L}\mu\|_{2/(2 - s)} \in L^{2}(\mathcal{G}(2,1))$ for every $s < t$. It is reasonable that the exponent $2/(2 - s)$ tends to infinity as $s,t \to 2$, but it is unsatisfactory that the exponent "$2$" in "$L^{2}(\mathcal{G}(2,1))$" stays constant. Indeed, for $t = 2$, trivially $\pi_{L}\mu \in L^{\infty}$ for \textbf{every} $L \in \mathcal{G}(2,1)$, or in other words $L \mapsto \|\pi_{L}\mu\|_{\infty} \in L^{\infty}(\mathcal{G}(2,1))$. Therefore, one would expect that there exists an exponent $p(s) \in [2,\infty)$ such that $p(s) \to \infty$ as $s \to 2$, and $L \mapsto \|\pi_{L}\mu\|_{p(s)} \in L^{p(s)}(\mathcal{G}(2,1))$ for every $s$-Frostman measure $\mu \in \mathcal{M}(\R^{2})$. This is a special case of the theorem below:
\begin{thm}\label{t:LpProjections} Let $0 < n < d$, and let $\mu \in \mathcal{M}(\R^{d})$ with $\spt \mu \subset B(1)$ satisfying the Frostman condition $\mu(B(x,r)) \leq C_{F}r^{s}$ for some $C_{F} \geq 1$, $s > n$, and for all balls $B(x,r) \subset \R^{d}$. Then,
\begin{equation}\label{form17} \int_{\mathcal{G}(d,n)} \|\pi_{V}\mu\|_{p}^{p} \, d\gamma_{d,n}(V) \lesssim_{d,p,s} C_{F}, \qquad 2 \leq p < \frac{2d - n - s}{d - s}. \end{equation}
\end{thm}

\begin{remark}\label{r:strichartz} Theorem \ref{t:LpProjections} can be viewed as an $L^{p}$ to $L^{p}$-Sobolev estimate for the $(d - n)$-plane transform, and there is plenty of existing literature on this topic. The most relevant reference is the paper \cite{MR782573} by Strichartz. Using complex interpolation between $H^{1}$ and $L^{2}$, he proves in \cite[Theorem 2.2]{MR782573} the following inequality for $f \in \mathcal{S}(\R^{d})$:
\begin{displaymath} \left( \int_{\mathcal{G}(d,n)} \|\pi_{V}f \|_{p,(d - n)/q}^{q} \, d\gamma_{d,n}(V) \right)^{1/q} \lesssim \|f\|_{L^{p}(\R^{d})}, \qquad 1 < p \leq 2. \end{displaymath}
Here $1/p + 1/q = 1$. This looks a little like \eqref{form17}, with two main differences: (i) we are interested in exponents $p > 2$, and (ii) we want to see the $L^{p}$-norm of $\pi_{V}\mu$ on the left hand side, instead of an $L^{p}$-Sobolev norm. The main reason why Strichartz' estimates are restricted to the range $1 < p \leq 2$ is that while the $(d - n)$-plane transform maps $L^{1}$ to $L^{1}$, and even $H^{1}$ to $H^{1}$, it fails to map $L^{\infty}$ to $L^{\infty}$. This would be the desirable right endpoint of interpolation in the range $2 \leq p < \infty$. We will (morally) fix the issue by considering a "localised" $(d - n)$-plane transform, which maps $L^{p}$ to $L^{p}$ for every $1 \leq p \leq \infty$: such localised estimates are good enough to yield information about compactly supported measures. The point (ii) is fairly minor: if $T$ is an operator which commutes with fractional Laplacians, such as the $(d - n)$-plane transform, then every estimate of the form $\|Tf\|_{p,\alpha} \leq C\|f\|_{p}$ implies an estimate of the form $\|Tf\|_{p} \leq C\|(-\bigtriangleup)^{-\alpha/2}f\|_{p}$. Eventually, the latter kind of estimate will be applied with $f = \mu$ to reach \eqref{form17}. \end{remark}

\subsubsection{Fractional Laplacians} The fractional Laplacian operator "$(-\bigtriangleup)^{s}$" already appeared in the discussion above, and will also be used extensively in the arguments below. For a thorough introduction, see \cite[Chapter V]{MR0290095}. Here we just mention the basic definitions, and the facts we will need. Let $\mathcal{S} := \mathcal{S}(\R^{d})$ be the space of Schwartz functions on $\R^{d}$, and let $f \in \mathcal{S}$. Then also $\hat{f} \in \mathcal{S}$. If $s \in \mathbb{C}$ with $\Rea s > -d/2$, the function 
\begin{equation}\label{rev1} \xi \mapsto (2\pi |\xi|)^{2s}\hat{f}(\xi) \end{equation}
is locally integrable, and has polynomial growth, so in particular it defines a tempered distribution. Here $r^{u + iv} = r^{u}r^{iv}$ for $r \geq 0$. By definition, $(-\bigtriangleup)^{s}f$ is the tempered distribution whose Fourier transform is the function defined in \eqref{rev1}. Thus,
\begin{displaymath} \widehat{(-\bigtriangleup)^{s}f} = (2\pi |\cdot|)^{2s}\hat{f}, \qquad f \in \mathcal{S}. \end{displaymath}
For $\Rea s \geq 0$, clearly $(2\pi |\cdot|)^{2s}\hat{f} \in L^{1} \cap L^{2}$ for $f \in \mathcal{S}$, so $(-\bigtriangleup)^{s}f$ is represented by a continuous $L^{2}$-function by Plancherel and the Fourier inversion theorem. For $s \in (0,d)$ and $f \in \mathcal{S}$, we will need to know that $(-\bigtriangleup)^{-s/2}f$ is the function represented by the \emph{Riesz potential}
\begin{equation}\label{rev3} V_{s}(f)(x) = c_{s} \int \frac{f(y) \, dy}{|x - y|^{d - s}}, \qquad x \in \R^{d}. \end{equation}
Here $c_{s} = \pi^{d/2}\Gamma(s/2)/\Gamma((d - s)/2) > 0$. This follows from \cite[Chapter V, Lemma 2]{MR0290095}. The function $V_{s}(f)$ is continuous if $f \in \mathcal{S}$ and $s \in (0,d)$.

Finally, we will need the following fact about $(-\bigtriangleup)^{iv}f$ for $v \in \R$:
\begin{equation}\label{rev2} \|(-\bigtriangleup)^{iv}f\|_{L^{p}(\R^{d})} \leq C_{p,v}\|f\|_{L^{p}(\R^{d})}, \qquad f \in \mathcal{S}(\R^{d}), \, 1 < p < \infty, \end{equation}
where $C_{p,v} \geq 1$ grows polynomially in $|v|$ (for $p \in (1,\infty)$ fixed). In fact, $f \mapsto (-\bigtriangleup)^{iv}f$ is a \emph{Calder\'on-Zygmund operator}. This follows from the H\"ormander-Mihlin multiplier theorem, see \cite[Theorem 5.2.7 + Example 5.2.9]{MR3243734}. In particular, $(-\bigtriangleup)^{iv}f \in L^{p}(\R^{d})$ for all $p \in (1,\infty)$, when $f \in \mathcal{S}$, and $v \in \R$. 

\subsection{Proof of Theorem \ref{t:LpProjections}} We then turn to the details of Theorem \ref{t:LpProjections}. It will be convenient to parametrise the projections $\pi_{V}\mu$ as follows. Let $\mathcal{O}(d)$ be the orthogonal group, and let $\pi_{0}(x_{1},\ldots,x_{d}) := (x_{1},\ldots,x_{n})$ be the projection to the $n$ first coordinates. Note that
\begin{displaymath} \pi_{0}^{\ast}(x_{1},\ldots,x_{n}) = (x_{1},\ldots,x_{n},0,\ldots,0) \in \R^{d}, \qquad (x_{1},\ldots,x_{n}) \in \R^{n}. \end{displaymath}
For a complex Borel measure $\mu$ on $\R^{d}$, and $\g \in \mathcal{O}(d)$, we define
\begin{displaymath} \pi_{\g}\mu := \pi_{0}(\g^{\ast}\mu), \end{displaymath} 
where $\g^{\ast}$ is the adjoint of $\g$ (or the inverse, since $\g^{\ast} = \g^{-1}$ for $\g \in \mathcal{O}(d)$). Of course the definition $\pi_{\g}\mu$ above also extends to functions $f \in L^{1}(\R^{d})$, and then $\pi_{\g}f \in L^{1}(\R^{n})$. We record the following useful formula for the Fourier transforms:
\begin{equation}\label{form20} \widehat{\pi_{\g}\mu}(\xi) = \hat{\mu}(\g \pi_{0}^{\ast}(\xi)) =: \hat{\mu}(\g \xi), \qquad \xi \in \R^{n}, \, \g \in \mathcal{O}(d). \end{equation}
The second equation means that we have identified $\xi \in \R^{n}$ and $\pi_{0}^{\ast}(\xi) \in \R^{d}$, and we will use this abbreviation in the sequel. 

It is very well-known that if $f \in \mathcal{S}(\R^{d})$, then the projections $\pi_{\g}f$ lie (quantitatively) in a certain homogeneous $L^{2}$-Sobolev space for almost every $\g \in \mathcal{O}(d)$. In fact:
\begin{equation}\label{sobolev} \int_{\mathcal{O}(d)} \|\pi_{\g} f\|_{2,(d - n)/2} \, d\g \lesssim \|f\|_{2}, \qquad f \in \mathcal{S}(\R^{d}). \end{equation}
This formula is essentially based on the Plancherel formula and the identity 
\begin{equation}\label{mattila} \int_{\mathcal{O}(d)} \int_{\R^{n}} |x|^{d - n}f(\g x) \, dx \, d\g = c(d,n)\int_{\R^{d}} f(x) \, dx, \qquad f \in L^{1}(\R^{d}), \end{equation}
see \cite[(24.2)]{MR3617376}. We will need a slight variant of \eqref{sobolev}, so we include the full details below:
\begin{lemma}\label{L2Lemma} Let $0 < n < d$, $\psi \in C^{\infty}_{c}(\R^{d})$, $z\in\mathbb{C}$ with $\Rea z \in [0,1]$, and let $T_{z}$ be the operator
\begin{equation}\label{Tz} T_{z}f(\g,x) := \pi_{\g}(\psi(-\bigtriangleup)^{(1 - z)(d - n)/4}f)(x), \qquad (\g,x) \in \mathcal{O}(d) \times \R^{n}, \end{equation}
defined for $f \in \mathcal{S}(\R^{d})$, and taking values in measurable functions on $\mathcal{O}(d) \times \R^{n}$. Then,
\begin{displaymath} \|T_{z}f\|_{L^{2}(\mathcal{O}(d) \times \R^{n})} \lesssim_{\psi,d,n} \|f\|_{L^{2}(\R^{d})}, \qquad f \in \mathcal{S}(\R^{d}),  \end{displaymath}
with bounds independent of $\Rea z \in [0,1]$. \end{lemma}

\begin{proof} Fix $f \in \mathcal{S}(\R^{d})$. Clearly $\psi(-\bigtriangleup)^{(1 - z)(d - n)/4}f \in C_{c}(\R^{d}) \subset L^{1}(\R^{d})$, so the Fourier transform formula \eqref{form20} is available. We write $h_{z}(\xi) := (2\pi |\xi|)^{(1 - z)(d - n)/2}$ for the symbol of $(-\bigtriangleup)^{(1 - z)(d - n)/4}$, and we abbreviate $\varphi := \hat{\psi}$. Then,
\begin{displaymath} \widehat{T_{z}f}(\g,\xi) = (\varphi \ast (h_{z}\hat{f}))(\g \xi), \qquad \xi \in \R^{n}, \, \g \in \mathcal{O}(d), \end{displaymath}
where $\widehat{T_{z}f}(\g,\xi)$ is the Fourier transform of $x \mapsto T_{z}f(\g,x) \in L^{1}(\R^{n})$. With this formula in hand, we may apply the Plancherel identity for every fixed $\g \in \mathcal{O}(d)$:
\begin{equation}\label{form26} \|T_{z}f\|_{L^{2}(\mathcal{O}(d) \times \R^{n})}^{2} = \int_{\mathcal{O}(d)} \int_{\R^{n}} |(\varphi \ast (h_{z}\hat{f}))(\g \xi)|^{2} \, d\xi \, d\g. \end{equation}
Next, we claim that if $f \in L^{2}(\R^{d})$ is arbitrary (and not only $f \in \mathcal{S}(\R^{d})$), then $\xi \mapsto (\varphi \ast (h_{z}\hat{f}))(\g \xi) \in L^{2}(\R^{n})$ for almost every $\g \in \mathcal{O}(d)$, and in fact
\begin{equation}\label{form25} \int_{\mathcal{O}(d)} \int_{\R^{n}} |(\varphi \ast (h_{z}\hat{f}))(\g\xi)|^{2} \, d\xi \, d\g \lesssim_{d,n,\psi} \|f\|_{L^{2}(\R^{d})}^{2}. \end{equation}
This follows from the next computation:
\begin{align*} \int_{\mathcal{O}(d)} \int_{\R^{n}} & |(\varphi \ast (h_{z}\hat{f}))(\g \xi)|^{2} \, d\xi \, d\g \lesssim_{\psi} \int_{\mathcal{O}(d)} \int_{\R^{n}} (|\varphi| \ast |h_{z}\hat{f}|^{2})(\g \xi) \, d\xi \, d\g\\
& = \int_{\R^{d}} |\varphi(y)| \int_{\mathcal{O}(d)} \int_{\R^{n}} (2\pi|\g \xi - y|)^{(1 - \Rea z)(d  - n)}|\hat{f}(\g \xi - y)|^{2} \, d\xi \, d\g \, dy\\
& \stackrel{\eqref{mattila}}{\sim_{d,n}} \int_{\R^{d}} |\varphi(y)| \int_{\R^{d}} |\xi|^{n - d}|\xi - y|^{(1 -\Rea z)(d - n)}|\hat{f}(\xi - y)|^{2} \, d\xi \, dy\\
& \stackrel{\xi \mapsto x + y}{=} \int_{\R^{d}} |\varphi(y)| \int_{\R^{d}} |x + y|^{n - d}|x|^{(1 - \Rea z)(d - n)}|\hat{f}(x)|^{2} \, dx \, dy\\
& = \int_{\R^{d}} |\hat{f}(x)|^{2}|x|^{(1 - \Rea z)(d - n)} \int_{\R^{d}} |\varphi(y)||x + y|^{n - d} \, dy \, dx \lesssim \int_{\R^{d}} |\hat{f}(x)|^{2} dx. \end{align*} 
The final inequality follows from the estimates $(1 - \Rea z)(d - n) \leq d - n$ and
\begin{displaymath} \int_{\R^{d}} |\varphi(y)||x + y|^{n - d} \, dy \lesssim_{\psi} |x|^{n - d}, \end{displaymath}
using the rapid decay of $\varphi = \hat{\psi}$, and recalling that $n < d$. In particular, a combination of \eqref{form26}-\eqref{form25} for $f \in \mathcal{S}(\R^{d})$ completes the proof of the lemma.  \end{proof}

By Lemma \ref{L2Lemma}, and the density of $\mathcal{S}(\R^{d}) \subset L^{2}(\R^{d})$, the operators 
\begin{displaymath} T_{z} \colon (\mathcal{S}(\R^{d}),\|\cdot\|_{L^{2}(\R^{d})}) \to L^{2}(\mathcal{O}(d) \times \R^{n}), \qquad \Rea z \in [0,1], \end{displaymath}
have unique extensions to operators $L^{2}(\R^{d}) \to L^{2}(\mathcal{O}(d) \times \R^{n})$. We keep denoting these operators with the same symbol $T_{z}$. We record that the extensions continue to have the following concrete representation: if $\Rea z \in [0,1]$, $f \in L^{2}(\R^{d})$, and $G \in L^{2}(\mathcal{O}(d) \times \R^{n})$, then
\begin{equation}\label{representation} \int_{\mathcal{O}(d) \times \R^{n}} (T_{z}f)(\g,x)G(\g,x) \, dx \, d\g = \int_{\mathcal{O}(d)} \int_{\R^{n}} (\varphi \ast h_{z}\hat{f})(\g \xi) \widehat{G}(\g,\xi) \, d\xi \, d\g. \end{equation} 
Indeed, by the definition of the "abstract" extension $T_{z} \colon L^{2}(\R^{d}) \to L^{2}(\mathcal{O}(d) \times \R^{n})$, if $\{f_{j}\}_{j \in \N} \subset \mathcal{S}(\R^{d})$ is a sequence of Schwartz functions converging to $f$ in $L^{2}(\R^{d})$, then
\begin{align*} \int_{\mathcal{O}(d) \times \R^{n}} (T_{z}f)(\g,x)G(\g,x) \, dx \, d\g & = \lim_{j \to \infty} \int_{\mathcal{O}(d)} \int_{\R^{n}} (T_{z}f_{j})(\g,x)G(\g,x) \, dx \, d\g\\
& = \lim_{j \to \infty} \int_{\mathcal{O}(d)} \int_{\R^{n}} (\varphi \ast h_{z}\widehat{f_{j}})(\g \xi) \widehat{G}(\g,\xi) \, d\xi \, d\g,  \end{align*}
where the final equation is due to Plancherel (for those a.e. $\g \in \mathcal{O}(d)$ such that $G(\g,\cdot) \in L^{2}(\R^{n})$). But then we may apply the inequality \eqref{form25} to the differences $f - f_{j} \in L^{2}(\R^{d})$ to conclude that the limit on the right equals the right hand side of \eqref{representation}.

Using the representation \eqref{representation}, it is not difficult to check (using Morera's theorem) that the family $\{T_{z}\}_{\Rea z \in [0,1]}$ is analytic in the usual sense that
\begin{displaymath} z \mapsto F_{f,G}(z) := \int_{\mathcal{O}(d) \times \R^{n}} T_{z}(f)(\g,x)G(\g,x) \, dx \, d\g = \int_{\mathcal{O}(d)} \int_{\R^{n}} (\varphi \ast h_{z}\hat{f})(\g \xi) \widehat{G}(\g,\xi) \, d\xi \, d\g \end{displaymath}
is analytic for $\Rea z \in (0,1)$, and continuous for $\Rea z \in [0,1]$, for all simple functions $f \colon \R^{d} \to \mathbb{C}$ and $G \colon \mathcal{O}(d) \times \R^{n} \to \mathbb{C}$ (continuity follows from dominated convergence, which is justified by repeating the estimates below \eqref{form25}). The map $F_{f,G}$ is also bounded for $\Rea z \in [0,1]$, as a consequence of the uniform $L^{2}(\R^{d}) \to L^{2}(\mathcal{O}(d) \times \R^{n})$-boundedness of the operators $T_{z}$. These are the hypotheses needed to apply Stein's interpolation theorem \cite{MR82586}, or see \cite[Theorem 1.3.7]{MR3243734} for a textbook reference. The details are contained in the next proposition.

\begin{proposition}\label{LpProp} Let $0 < n < d$, $2 \leq p < \infty$, and $(p - 2)/p < \theta \leq 1$. Then, the operator $T_{\theta}$ has a bounded extension to $L^{p}(\R^{d})$. More precisely, if $f \in L^{2}(\R^{d}) \cap L^{p}(\R^{d})$, then $\|T_{\theta}f\|_{L^{p}(\mathcal{O}(d) \times \R^{n})} \lesssim_{p,\theta} \|f\|_{L^{p}(\R^{d})}$. \end{proposition}

\begin{proof} Fix $2 \leq p < \infty$ and $(p - 2)/p < \theta \leq 1$. Then, define $p_{\infty} \in [p,\infty)$ as the solution to
\begin{displaymath} \frac{1}{p} = \frac{1 - \theta}{2} + \frac{\theta}{p_{\infty}}. \end{displaymath} 
Note that if $p$ and $\theta$ are related as above, then $\theta = (p_{\infty}/p) \cdot (p - 2)/(p_{\infty} - 2)$, and this expression takes all values on the interval $((p - 2)/p,1]$ as $p_{\infty}$ ranges in $[p,\infty)$. 

We write $\overline{T}_{z} := e^{z^{2}} \cdot T_{z}$. Since $z \mapsto e^{z^{2}}$ is a bounded analytic function on $\Rea z \in [0,1]$, the operators $\overline{T}_{z}$ have all the good properties of the operators $T_{z}$, but this (standard) trick helps to establish the following: the operators $\overline{T}_{1 + ir}$ are uniformly bounded $L^{2}(\R^{d}) \cap L^{p_{\infty}}(\R^{d}) \to L^{p_{\infty}}(\mathcal{O}(d) \times \R^{n})$ for $r \in \R$. We first verify this for Schwartz functions, so fix $f \in \mathcal{S}(\R^{d})$. Then we have the explicit expression \eqref{Tz} for the operators $T_{1 + ir}$, which allows us to estimate as follows:
\begin{align*} \|\overline{T}_{1 + ir}f\|_{L^{p_{\infty}}(\mathcal{O}(d) \times \R^{n})}^{p_{\infty}} & \leq e^{(1-r^{2})p_{\infty}} \int_{\mathcal{O}(d)} \|\pi_{\g}(\psi (-\bigtriangleup)^{-ir(d - n)/4}f)\|_{L^{p_{\infty}}(\R^{n})}^{p_{\infty}} \, d\g\\
& \lesssim_{\psi} e^{(1-r^{2})p_{\infty}}  \|(-\bigtriangleup)^{-ir(d - n)/4}f\|_{L^{p_{\infty}}(\R^{d})}^{p_{\infty}}\\
& \lesssim_{p_{\infty}} \mathrm{poly}(|r|) \cdot e^{-r^{2}p_{\infty}} \|f\|_{L^{p_{\infty}}(\R^{d})}^{p_{\infty}} \lesssim \|f\|_{L^{p_{\infty}}(\R^{d})}^{p_{\infty}}. \end{align*} 
The "localisation" by the fixed bump function $\psi \in C^{\infty}_{c}(\R^{d})$ was crucial to pass from the first line to the second: the maps $f \mapsto \pi_{\g}f$ are not bounded $L^{p}(\R^{d}) \to L^{p}(\R^{n})$ for any $p > 1$, but the maps $f \mapsto \pi_{\g}(\psi f)$ are bounded on all $L^{p}$-spaces by an application of H\"older's inequality. As another remark, the "$\mathrm{poly}(|r|)$" factor reflects the $L^{p_{\infty}}(\R^{d}) \to L^{p_{\infty}}(\R^{d})$ boundedness of the imaginary fractional Laplacian $(-\bigtriangleup)^{-ir(d - n)/4}$, recall \eqref{rev2}. The mitigation of this factor was the only reason to introduce the factor $e^{z^{2}}$.

It remains to argue that the same estimate holds for $f \in L^{2}(\R^{d}) \cap L^{p_{\infty}}(\R^{d})$. Pick a sequence $\{f_{i}\}_{i \in \N} \subset \mathcal{S}(\R^{d})$ which converges to $f$ in both $L^{2}(\R^{d})$ and $L^{p_{\infty}}(\R^{d})$. Then, for $r \in \R$, the functions $T_{1 + ir}(f_{i})$ converge to $T_{1 + ir}(f)$ in $L^{2}(\mathcal{O}(d) \times \R^{n})$, so after passing to a subsequence, we may assume that $T_{1 + ir}(f_{i}) \to T_{1 + ir}(f)$ almost everywhere. Then, by Fatou's lemma,
\begin{align} \int_{\mathcal{O}(d)} \int_{\R^{n}} & |(T_{1 + ir}f)(\g,x)|^{p_{\infty}} \, dx \, d\g \notag\\
&\label{form27} \leq \liminf_{i \to \infty} \|T_{1 + ir}(f_{i})\|^{p_{\infty}}_{L^{p_{\infty}}(\mathcal{O}(d) \times \R^{n})} \lesssim_{p_{\infty}} \liminf_{i \to \infty} \|f_{i}\|^{p_{\infty}}_{L^{p_{\infty}}(\R^{d})} = \|f\|^{p_{\infty}}_{L^{p_{\infty}}(\R^{d})}. \end{align}
Hence $T_{1 + ir}f \in L^{p_{\infty}}(\mathcal{O}(d) \times \R^{n})$, and $\|T_{1 + ir}f\|_{L^{p_{\infty}}(\mathcal{O}(d) \times \R^{n})} \lesssim_{p_{\infty}} \|f\|_{L^{p_{\infty}}(\R^{d})}$.

We have now verified all the hypotheses of Stein's interpolation theorem, as stated in \cite[Theorem 1.3.7]{MR3243734}, for the operator family $\{\overline{T}_{z}\}_{\Rea z \in [0,1]}$. The conclusion is that 
\begin{displaymath} \|T_{\theta}f\|_{L^{p}(\mathcal{O}(d) \times \R^{n})} \leq \|\overline{T}_{\theta}f\|_{L^{p}(\mathcal{O}(d) \times \R^{n})} \lesssim_{p_{\infty}} \|f\|_{L^{p}(\R^{d})} \end{displaymath}
for all simple functions $f$ on $\R^{d}$. Since the choice of $p_{\infty}$ only depends on $p,\theta$, the notation $\lesssim_{p_{\infty}}$ is equivalent to $\lesssim_{p,\theta}$. The extension of the bound above for $f \in L^{2}(\R^{d}) \cap L^{p}(\R^{d})$ follows as in \eqref{form27}, so the proof of the proposition is complete. \end{proof}

We are then ready to prove Theorem \ref{t:LpProjections}.

\begin{proof}[Proof of Theorem \ref{t:LpProjections}] Let $n < s \leq d$, and let $\mu \in \mathcal{M}(\R^{d})$ satisfy the assumptions of the theorem: $\spt \mu \subset B(1) \quad \text{and} \quad \mu(B(x,r)) \leq C_{F}r^{s}$ for some constant $C_{F} > 0$, and for all balls $B(x,r) \subset \R^{2}$. We assume\footnote{That is, we convolve $\mu$ with an approximate identity $\varphi_\delta$, so that the resulting function is $C^\infty(\R^d)$. Obviously, our estimates will not depend on $\delta$. For notation's sake, we will not make this explicit, and we will simply make the qualitative assumption above.} in addition (qualitatively) that $\mu \in C^{\infty}(\R^{d})$. Let $\psi \in C^{\infty}_{c}(\R^{d})$ be a function satisfying $\mathbf{1}_{B(1)} \leq \psi \leq \mathbf{1}_{B(2)}$, so $\mu = \psi \mu$. We abbreviate $\varphi := \hat{\psi} \in \mathcal{S}(\R^{d})$.

Now, fix $2 \leq p < (2d - n - s)/(d - s)$ and $\epsilon \in (0,1)$, where $\epsilon$ is chosen sufficiently small so to satisfy the hypotheses of Proposition \ref{prop:Riesz} below (it will then only depend on $d,p,s$, as per Proposition \ref{prop:Riesz}). Then set 
\begin{displaymath} \alpha := (1 - \epsilon)\frac{d - n}{p} < \frac{d - n}{p}. \end{displaymath}
The rationale for this choice of "$\alpha$" will be that if "$\theta$" solves $(1 - \theta)(d - n)/2 = \alpha$, then 
\begin{equation}\label{theta} \theta = \frac{p - 2}{p} + \frac{2\epsilon}{p} \quad \Longrightarrow \quad \frac{p - 2}{p} < \theta < 1, \end{equation}
and Proposition \ref{LpProp} will be applicable with this "$\theta$". Note also that $(2\pi|\xi|)^{\alpha} = h_{\theta}(\xi)$ with the notation used in formula \eqref{representation}.

Let $q\ge 1$ be such that $\frac{1}{p}+\frac{1}{q}=1$. Fixing also a simple function $G \colon \mathcal{O}(d) \times \R^{n} \to \mathbb{C}$ with $\|G\|_{L^{q}(\mathcal{O}(d) \times \R^{n})} \leq 1$, we write
\begin{align*} \int_{\mathcal{O}(d)} \int_{\R^{n}} (\pi_{\g}\mu)(x) G(\g,x) \, dx \, d\g & = \int_{\mathcal{O}(d)} \int_{\R^{n}} (\pi_{\g}(\psi \mu))(x) G(\g,x) \, dx \, d\g\\
& = \int_{\mathcal{O}(d)} \int_{\R^{n}} (\varphi \ast \hat{\mu})(\g\xi) \widehat{G}(\g,\xi) \, d\xi \, d\g\\
& = \int_{\mathcal{O}(d)} \int_{\R^{n}} (\varphi \ast h_{\theta}\widehat{V_{\alpha}(\mu)})(\g\xi)\widehat{G}(\g,\xi) \, d\xi \, d\g, \end{align*} 
where 
\begin{displaymath} V_{\alpha}(\mu)(x) = (-\bigtriangleup)^{-\alpha/2}\mu(x) = c_{\alpha}\int \frac{\mu(y) \, dy}{|x - y|^{d - \alpha}}, \qquad x \in \R^{d}, \end{displaymath}
is the Riesz potential of $\mu$ with index $\alpha$, recall \eqref{rev3}. Note that 
\begin{displaymath} \alpha = (1 - \epsilon)\frac{d - n}{p} \text{ and } p \geq 2 \quad \Longrightarrow \quad d - \alpha = \frac{d(p - 1+\epsilon) + (1-\epsilon)n}{p} \geq \frac{d + n}{2} > \frac{d}{2}, \end{displaymath}
so the smoothness and compact support of $\mu$ imply $V_{\alpha}(\mu)(x) \leq O((1 + |x|)^{-d/2 - \kappa})$ for some $\kappa > 0$, assuming that $\epsilon > 0$ in the definition of "$\alpha$" is chosen sufficiently small. In particular, $V_{\alpha}(\mu) \in L^{2}(\R^{d})$. This permits us to use the representation formula \eqref{representation} for the operator $T_{\theta}$ with the choices $f := V_{\alpha}(\mu)$ and "$\theta$" as in \eqref{theta}:
\begin{displaymath} \int_{\mathcal{O}(d)} \int_{\R^{n}} (\pi_{\g}\mu)(x) G(\g,x) \, dx \, d\g = \int_{\mathcal{O}(d) \times \R^{n}} T_{\theta}(V_{\alpha}(\mu))(\g,x)G(\g,x) \, dx \, d\g. \end{displaymath}
The operator $T_{\theta}$ is bounded $L^{2}(\R^{d}) \cap L^{p}(\R^{d}) \to L^{p}(\mathcal{O}(d) \times \R^{n})$ for this "$\theta$" by Proposition \ref{LpProp}, so we conclude that
\begin{displaymath} \left| \int_{\mathcal{O}(d)} \int_{\R^{n}} (\pi_{\g}\mu)(x) G(\g,x) \, dx \, d\g \right| \lesssim \|V_{\alpha}(\mu)\|_{L^{p}(\R^{d})}\|G\|_{L^{q}(\mathcal{O}(d) \times \R^{n})} \leq \|V_{\alpha}(\mu)\|_{L^{p}(\R^{d})}. \end{displaymath}
The proof of Theorem \ref{t:LpProjections} is now completed by showing that $\|V_{\alpha}(\mu)\|_{L^{p}(\R^{d})} \lesssim_{d,p,s} C_{F}$ with the choice $\alpha = (1 - \epsilon)(d - n)/p$, if $\epsilon > 0$ small enough, depending on $d,p,s$. This follows from \cite[(3.1)]{MR2994685}, but that argument is based on interpolation, and we give an elementary proof in Proposition \ref{prop:Riesz} for completeness. This concludes the proof of Theorem \ref{t:LpProjections}. \end{proof}

\begin{proposition}\label{prop:Riesz} Let $d \geq 2$, $n \geq 1$, $n < s \leq d$, and let $\mu \in \mathcal{M}(\R^{d})$ satisfy $\mu(B(x,r)) \leq C_{F}r^{s}$ for all balls $B(x,r) \subset \R^{d}$, and $\spt \mu \subset B(1)$. Let $2 \leq p < (2d - n - s)/(d - s)$. Then, if $\epsilon \in (0,1)$ is small enough, depending only on $d,p,s$, and $\alpha := (1 - \epsilon)(d - n)/p$, we have
\begin{equation}\label{form14} \|V_{\alpha}(\mu)\|_{p} \sim_{\alpha} \left[ \int \left( \int \frac{d\mu(y)}{|x - y|^{d - \alpha}} \right)^{p} \, dx \right]^{1/p} \lesssim_{d,p,s} C_{F}. \end{equation}
\end{proposition}

\begin{proof} Fix $2 \leq p < (2d - s - n)/(d - s)$. Fix also $x \in \R^{d}$ and $\epsilon > 0$ (whose value will eventually depend on $d,p,s$), and start by decomposing the inner integral as
\begin{displaymath} \left( \int \frac{d\mu(y)}{|x - y|^{d - \alpha}} \right)^{p} \lesssim \left( \sum_{j \geq 0} 2^{j(d - \alpha)} \mu(B(x,2^{-j + 2})) \right)^{p} \lesssim_{\epsilon,p} \sum_{j \geq 0} 2^{j(dp + \epsilon - \alpha p)} \mu(B(x,2^{-j + 2}))^{p}. \end{displaymath}
The second inequality is a consequence of H\"older's inequality with exponent $p > 1$, after introducing artificially the factors $2^{\epsilon j/p}$ and $2^{-\epsilon j/p}$. The choice of $\epsilon > 0$ will eventually just depend on $d,p,s$, so "$\lesssim_{\epsilon}"$ means the same as "$\lesssim_{d,p,s} 1$". We may restrict to indices $j \geq 0$ by the assumption $\spt \mu \subset B(1)$. Plugging the inequality above to the left hand side of \eqref{form14} yields
\begin{equation}\label{form16} \int \left( \int \frac{d\mu(y)}{|x - y|^{d - \alpha}} \right)^{p} \, dx \lesssim_{\epsilon,p} \sum_{j \geq 0} 2^{j(dp + \epsilon - \alpha p)} \int \mu(B(x,2^{-j + 2}))^{p} \, dx. \end{equation} 
To treat the remaining integral, we make the following claim, for $\delta = 2^{-j + 2} \in 2^{-\N}$:
\begin{equation}\label{form15} \int \mu(B(x,\delta))^{p} \, dx \lesssim_{d,p} C_{F}^{p} \cdot \delta^{d - s + ps}. \end{equation}
To prove \eqref{form15}, we decompose $\mu$ as follows: for $i \geq 0$, let $\mathcal{Q}_{i} \subset \overline{\mathcal{D}}_{\delta}(\R^{d})$ be the collection of those closed dyadic $\delta$-cubes with the property
\begin{displaymath} 2^{-i - 1} \cdot C_{F}\delta^{s} \leq \mu(Q) \leq 2^{-i} \cdot C_{F}\delta^{s}, \qquad Q \in \mathcal{Q}_{i}. \end{displaymath}
Further, let $\mu_{i}$ be the restriction of $\mu$ to $\cup \mathcal{Q}_{i}$. Clearly $\mu \leq \sum_{i \geq 0} \mu_{i}$, and $\mu_{i}(B(x,\delta)) \lesssim 2^{-i} \cdot C_{F}\delta^{s}$ for all $x \in \R^{d}$. For $\epsilon > 0$ arbitrary, it follows that
\begin{align*} \int \mu(B(x,\delta))^{p} \, dx & \leq \int \left( \sum_{i \geq 0} \mu_{i}(B(x,\delta)) \right)^{p} \, dx\\
& \lesssim_{\epsilon,p} \sum_{i \geq 0} 2^{i\epsilon} \int \mu_{i}(B(x,\delta))^{p} \, dx\\
& \lesssim C_{F}^{p} \cdot \delta^{ps} \cdot \sum_{i \geq 0} 2^{i(\epsilon - p)} \cdot \calH^d(\{x \in \R^{d} : B(x,\delta) \cap \spt \mu_{i} \neq \emptyset\}).  \end{align*} 
Recall that $\spt \mu_{i}$ consists of the union of the cubes $Q \in \overline{\mathcal{D}}_{\delta}(\R^{d})$, which satisfy $\mu(Q) \sim 2^{-i} \cdot C_{F}\delta^{s}$. Since $\|\mu\| = \mu(B(1)) \lesssim C_{F}$, we have $\card \mathcal{Q}_{i} \lesssim 2^{i} \cdot \delta^{-s}$, and consequently
\begin{displaymath}  \mathcal{H}^{d}(\{x \in \R^{d} : B(x,\delta) \cap \spt \mu_{i} \neq \emptyset\}) \lesssim \delta^{d} \cdot (\card \mathcal{Q}_{i}) \lesssim 2^{i} \cdot \delta^{d - s}. \end{displaymath} 
Therefore, since $1 + \epsilon - p < 0$ (recall that $p \geq 2$), we have
\begin{displaymath} \int \mu(B(x,\delta))^{p} \, dx \lesssim_{\epsilon,p} C_{F}^{p} \cdot \delta^{d - s + ps} \cdot \sum_{i \geq 0} 2^{i(1 + \epsilon - p)} \lesssim_{p} C_F^p\cdot\delta^{d - s + ps}, \end{displaymath}
as claimed in \eqref{form15}. 

Inserting the inequality \eqref{form15} into \eqref{form16} now yields
\begin{displaymath} \int \left( \int \frac{d\mu(y)}{|x - y|^{d - \alpha}} \right)^{p} \, dx \lesssim_{d,\epsilon,p} C_{F}^{p} \cdot \sum_{j \geq 0} 2^{j(dp + \epsilon - \alpha p - d + s - ps)}. \end{displaymath}
The geometric series is summable if and only if $dp + \epsilon - \alpha p - d + s - ps < 0$. Recalling that $\alpha = (1 - \epsilon)(d - n)/p$, this amounts to 
\begin{displaymath} p < \frac{(1 - \epsilon)(d - n) + d - s - \epsilon}{d - s}. \end{displaymath}
Since we assumed that $p < (2d - n - s)/(d - s)$, this is true with $\epsilon > 0$ small enough, depending only on $d,p,s$. \end{proof}

\section{The incidence estimate}\label{s:incidence}
 In this section we prove Theorem \ref{t:incidences}, which we recall. 

\begin{thm} Let $0 < n < d$ and $C, C_F \geq 1$. Let $\mathcal{V} \subset \mathcal{A}(d,n)$ be a $\delta$-separated set of $n$-planes, and let $P \subset B(1) \subset \R^{d}$ be a $\delta$-separated $(\delta,t,C_F)$-set with $t > d - n$. For $r>0$ let $\mathcal{I}_{r}(P,\mathcal{V}) = \{(p,V) \in P \times \mathcal{V} : p \in V(r)\}$. Then, for every $\varepsilon>0$ we have
	\begin{equation}\label{main-inc-eq}
		|\mathcal{I}_{C\delta}(P,\mathcal{V})| \lesssim_{C,d,\varepsilon,t} \delta^{-\varepsilon} \cdot C_{F} \cdot |P| \cdot |\mathcal{V}|^{n/(d+n-t)} \cdot \delta^{n(t+1-d)(d-n)/(d+n-t)}.
	\end{equation}
	\end{thm}

\subsubsection*{Pigeonholing}
We start off by finding subfamilies $P_1$ and $\calV_1$ which have a uniform number of incidences. 
For $V \in \calA(d,n)$, set $N_V:=|P \cap V(C\delta)|$. Note that since $P \subset B(1)$ is $\delta$-separated, we have $N_{V} \lesssim \delta^{-d}$ for every $V \in \mathcal{A}(d,n)$. By the pigeonhole principle, there exists a number $N \in \N$ and a subfamily $\calV_1 \subset \calV$ such that
\begin{align}\label{e:pigeon1}
	\tfrac{N}{2} \leq N_V \leq N \mbox{ for all } V \in \calV_1, \quad \mbox{ and } \quad N \cdot |\calV_1| \approx |\mathcal{I}_{C\delta}(P, \calV)|.
	\end{align}
The implicit constants behind the "$\approx$" notation here are allowed to depend on "$d$". For $p \in P$, set
\begin{equation}\label{to3}
	M_p := |\calV^p|:= |\{V \in \calV_1 :  p \in V(C\delta)\}|.
\end{equation}
Using the pigeonhole principle once more, we find a number $M \in \N$ and a subfamily $P_1 \subset P$ so that
\begin{align}\label{e:pigeon2}
	\tfrac{M}{2} \leq M_p \leq M \mbox{ for all } p \in P_1, \quad \mbox{ and } \quad M \cdot |P_1| \approx |\mathcal{I}_{C\delta}(P, \calV)|. 
	\end{align}
		
\subsubsection*{Lower bounds for radial projections}

Later on, we will apply Theorem \ref{mainIntro} to the following density:
	\begin{equation} \label{e:mu}
		\mu(y) := \frac{1}{|P|} \sum_{p \in P} \varphi_{\delta} (p - y), \qquad y \in \R^{d}.
	\end{equation}
	Here $\varphi_{\delta} = (C\delta)^{-d}\varphi(\cdot/(C\delta)) \in C^{\infty}_{c}(\R^{d})$ is a non-negative radial function satisfying $\varphi_\delta(x) = (C\delta)^{-d}$ for $x \in B(3C\delta)$, $\spt \varphi_\delta \subset B(4C\delta)$, and $\mathrm{Lip}(\varphi_{\delta})\le (C\delta)^{-d-1}$. We will abuse notation and denote by $\mu$ also the measure given by the density above. It is easy to check that $\mu(\R^{d}) \sim 1$, and also it follows from the $(\delta,t,C_F)$-set property of $P$ that $\mu$ is a $t$-Frostman measure with constant $\sim C_F$, i.e.  $\mu(B(x,r))\lesssim C_F r^t$ for all $x\in\R^d$ and $r>0$. 
	
	Now fix $x \in \R^{d}$. Since $\mu$ has continuous density, we may define another continuous density $\mu_{x}$ on $\mathcal{G}(d,n)$ by the following formula:
	\begin{equation}
		\label{e:mux} \mu_{x}(\mathbf{V}) := \int_{x + \mathbf{V}} \mu(y) \, d \calH^n(y), \qquad \mathbf{V} \in \mathcal{G}(d,n). 	
	\end{equation}
In this section, we will keep the notational convention that affine $n$-planes are denoted $V,V'$ and $n$-dimensional subspaces $\mathbf{V},\mathbf{V}'$. For every $V \in \calV^p$, as in \eqref{to3}, there exists a unique $n$-dimensional subspace $\mathbf{V} \in \mathcal{G}(d,n)$ and a point $x_V \in B(p,C\delta)$ so that $V= \mathbf{V} + x_V$. While the subspaces $\mathbf{V} \in \mathcal{G}(d,n)$ obtained in this way need not be $\delta$-separated, it is easy to find $(\delta/2)$-separated subset of cardinality comparable to $|\calV^p| \sim M$. 
\begin{lemma}\label{l:T0}
	For every $p \in P_{1}$, there exists a $(\delta/2)$-separated subset $\calV_0^p \subset \{\mathbf{V} : \mathbf{V} + x_{V} \in \calV^p\}$ such that $|\calV_0^p| \sim_{d} |\calV^p| \sim M$. 
\end{lemma}

We leave the details to the reader, and turn to proving a lower bound for the integral of the density "$\mu$" along certain (affine) $n$-planes:

\begin{lemma}\label{l:low-bound1} Let $x \in P_1(\delta/10)$, so $|x - p| \leq \delta/10$ for some $p \in P_1$. Let $\mathbf{V} \in \calV_0^p$, and $\mathbf{V}' \in B(\mathbf{V},\delta/10) \subset \mathcal{G}(d,n)$. Then, 
\begin{align}\label{form57}
    \int_{\mathbf{V}' + x} \mu(y) \, d\calH^n(y) \gtrsim_{d} N \cdot \frac{\delta^{n - d}}{C|P|}.
\end{align}
\end{lemma}

\begin{figure}[h!]
\begin{center}
\begin{overpic}[scale = 1.6]{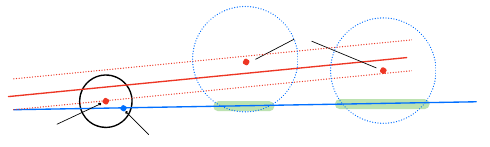}
\put(9,3){$p$}
\put(30,0){$x$}
\put(15,1){$B(p,C\delta)$}
\put(61.5,22){$p'$}
\put(37,3){$(\mathbf{V}' + x) \cap B(p',3C\delta)$}
\put(85,18){$\mathbf{V} + x_{V} \in \calV^p$}
\end{overpic}
\caption{The proof of Lemma \ref{l:low-bound1}.}\label{IncidenceFig}
\end{center}
\end{figure}
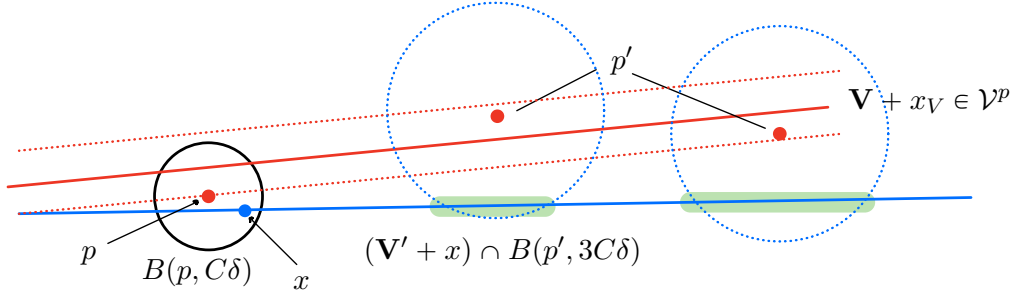

\begin{proof} The proof is depicted in Figure \ref{IncidenceFig}. By definition of $\mathbf{V} \in \calV_{0}^p$, there exists a vector $x_{V} \in B(p,C\delta)$ such that $\mathbf{V} + x_{V} \in \calV^p$. This plane is drawn in red. Since $\mathbf{V} + x_{V} \in \calV^p \subset \mathcal{V}_{1}$, recall \eqref{e:pigeon1}, the $C\delta$-neighbourhood $(\mathbf{V} + x_{V})(C\delta)$ contains a subset $P_{V} \subset P$ with $|P_{V}| = N_{V} \sim N$. Two elements of $P_{V}$ are drawn in red. The density "$\mu$" then satisfies 
\begin{equation}\label{to4} \mu(y) \gtrsim (C|P|\delta^{d})^{-1}, \qquad y \in B(p',3C\delta), \, p' \in P_{V}, \end{equation}
by the definition of $\mu$ in \eqref{e:mu}. Finally, if $\mathbf{V}' \in B(\mathbf{V},\delta/10)$ and $x \in B(p,\delta/10)$ (as in the statement), then the plane $\mathbf{V}' + x$, drawn in blue, remains close to $\mathbf{V} + x_{V}$ inside $B(1)$: in particular
\begin{equation}\label{to5} \mathcal{H}^{n}((\mathbf{V}' + x) \cap B(p',3C\delta)) \gtrsim_{d} \delta^{n}, \qquad p' \in P_{V}. \end{equation}
Two of the intersections $(\mathbf{V}' + x) \cap B(p',3C\delta)$ are drawn in green. Now \eqref{form57} follows by combining \eqref{to4}-\eqref{to5}, and recalling that $|P_{V}| \sim N$.  \end{proof}

\begin{lemma} Let $x \in P_1(\delta/10)$, let $\mu$ be as in \eqref{e:mu} and $\mu_x$ be as in \eqref{e:mux}. Then, for $q\geq 1$, 
	\begin{equation}\label{e:lowerbound}
		\|\mu_x\|_{L^q(\mathcal{G}(d,n))}^q \gtrsim_{d} M \cdot \delta^{n(d-n)} \left(N\cdot \frac{\delta^{n-d}}{C|P|}\right)^q.
	\end{equation}
	
	\end{lemma}
	
	\begin{proof}
	Fix $x \in P_{1}(\delta/10)$. By definition,
	\begin{equation}\label{form55}	\|\mu_x\|_{L^q(\mathcal{G}(d,n))}^q = \int_{\mathcal{G}(d,n)} \left| \int_{\mathbf{V}' + x} \mu(y) \, d\calH^n(y) \right|^q \, d\gamma_{d,n}(\mathbf{V}').
	\end{equation}
We will use the well-known fact, see \cite[Proposition 4.1]{MR3044214}, that 
\begin{align}\label{e:form32}
    \gamma_{d,n} (B(\mathbf{V},r)) \gtrsim_{d} r^{n(d - n)}, \qquad \mathbf{V} \in \mathcal{G}(d,n), \, 0 < r \leq 1.
\end{align}
Since $x \in P_{1}(\delta/10)$, we may find $p \in P_{1}$ with $|x - p| \leq \delta/10$. Recall from Lemma \ref{l:T0} that $|\calV_{0}^p| \sim M$, and the subspaces in $\calV_{0}^p$ are $(\delta/2)$-separated, so in particular the balls $B(\mathbf{V},\delta/10)$ with $\mathbf{V} \in \calV_{0}^p$ are disjoint. We may then estimate the right hand side of \eqref{form55}:
\begin{align*}
    \eqref{form55} & \geq \sum_{\mathbf{V} \in \calV_0^p} \int_{B(\mathbf{V},\delta/10)} \left| \int_{\mathbf{V}' + x} \mu(y) \, d\calH^n(y) \right|^q \, d\gamma_{d,n}(\mathbf{V}') \nonumber\\
    & \stackrel{\eqref{form57}}{\gtrsim_{d}} \sum_{\mathbf{V} \in \calV_0^p} \gamma_{d,n}(B(\mathbf{V},\tfrac{\delta}{10})) \cdot \left(N\cdot  \frac{\delta^{n-d}}{C|P|}\right)^q  \stackrel{\eqref{e:form32}}{\gtrsim_{d}} M \cdot \delta^{n(d-n)} \left(N \cdot \frac{\delta^{n-d}}{C|P|}\right)^q.
    \end{align*}
This proves the lemma. 
	\end{proof}

\subsubsection*{Upper bounds for radial projections} During the remainder of the section, we will write $V,V'$ for elements of $\mathcal{G}(d,n)$, since elements of $\mathcal{A}(d,n)$ no longer appear here. The following identity is useful for computing an upper bound for the $L^q$ norm of $\mu_x$. In the planar case, this is essentially \cite[Lemma 3.1]{MR3892404}. 
\begin{lemma}
Let $q \geq 1$. With the notation as above, 
\begin{equation}\label{form4} 
\int \|\mu_{x}\|_{L^{q}(\mathcal{G}(d,n))}^{q} \, d\mu(x) = \int_{\mathcal{G}(d,n)} \|\pi_{V^\perp}\mu\|_{L^{q + 1}(V^\perp)}^{q + 1} \, d\gamma_{d,n}(V). 
\end{equation}
\end{lemma}

\begin{proof}
Let $V \in \mathcal{G}(d,n)$. Since $\mu \in C_{c}(\R^{d})$, also the push-forward measure $\pi_{V^{\perp}}\mu$ has a continuous compactly supported density on $V^{\perp}$, and 
\begin{align}\label{e:form10}
\mu_x(V) = \int_{x + V} \mu(y) \, d\mathcal{H}^{n}(y) = (\pi_{V^{\perp}}\mu)(\pi_{V^{\perp}}(x)), \qquad x \in \R^{d}.
\end{align}
Writing $x = \pi_{V}(x) + \pi_{V^{\perp}}(x) = v + v^{\perp}$ for a fixed plane $V \in \mathcal{G}(d,n)$, and using Fubini's theorem in $\R^{d} = V \times V^{\perp}$, we may now compute as follows:

\begin{align*}  &  \int \|\mu_x\|_{L^q(\mathcal{G}(d,n))}^q \, d\mu(x)  
    \stackrel{\eqref{e:form10}}{=} \int \int_{\mathcal{G}(d,n)} (\pi_{V^{\perp}}\mu)(\pi_{V^{\perp}}(x))^{q} \, d \gamma_{d,n}(V) \, d\mu(x) \\
    & = \int_{\mathcal{G}(d,n)} \int_{V^{\perp}} \int_{V} (\pi_{V^{\perp}}\mu)(v^{\perp})^{q} \mu(v + v^{\perp}) \, d\mathcal{H}^{n}(v) \, d\mathcal{H}^{d - n}(v^{\perp}) \, d\gamma_{d,n}(V)\\
    & = \int_{\mathcal{G}(d,n)} \int_{V^{\perp}} (\pi_{V^{\perp}}\mu)(v^{\perp})^{q} \left(\int_{V} \mu(v + v^{\perp}) \, d\mathcal{H}^{n}(v) \right) \, d\mathcal{H}^{d - n}(v^{\perp}) \, d\gamma_{d,n}(V)\\
    & = \int_{\mathcal{G}(d,n)} \int_{V^{\perp}} (\pi_{V^{\perp}}\mu)(v^{\perp})^{q + 1} \, d\mathcal{H}^{d - n}(v^{\perp}) \, d\gamma_{d,n}(V) = \int_{\mathcal{G}(d,n)} \|\pi_{V^{\perp}}\mu\|_{L^{q + 1}(V^{\perp})}^{q + 1} \, d\gamma_{d,n}(V). \end{align*}
This completes the proof of the lemma.
\end{proof}

We are now ready to prove Theorem \ref{t:incidences}. 
\begin{proof}[Proof of Theorem \ref{t:incidences}]
Let $g: \mathcal{G}(d, d-n) \to \R$ be the map $W \mapsto \|\pi_W \mu\|_{L^{q+1}}^{q+1}(W)$, and let $f: \mathcal{G}(d,n) \to \mathcal{G}(d,d-n)$ be the map which sends $V$ to its orthogonal complement $W=V^\perp \in \mathcal{G}(d,d-n)$. Then we can rewrite the right hand side of \eqref{form4} as 
\begin{align}\label{e:form40}
\int_{\mathcal{G}(d,n)} (g \circ f)(V) \, d\gamma_{d,n}(V) & = \int_{\mathcal{G}(d,d-n)} g(W) \, d (f \gamma_{d,n}) (W) \nonumber \\
& = \int_{\mathcal{G}(d,d-n)} \|\pi_{W}\mu\|_{L^{q+1}(W)}^{q+1} \, d \gamma_{d, d-n}(W).
\end{align}
In the last equality we used the fact that $f \gamma_{d,n}$ defines an $\mathcal{O}(d)$-invariant probability measure on $\mathcal{G}(d,d-n)$, so $f \gamma_{d,n}= \gamma_{d,d-n}$ (see \cite[(3.10)]{zbMATH01249699}). 

Recall that the density $\mu$ defines a Radon measure satisfying the $t$-Frostman condition with constant $\sim C_F$, that is, $\mu \in \mathcal{M}_{t}$ and $\mu(B(x,r))\lesssim C_F r^t$ for all $x\in\R^d$ and $r>0$. Hence, from Theorem \ref{mainIntro} we find that the integral on the right hand side of \eqref{e:form40} is finite whenever
	\begin{displaymath}  q+1 < \tfrac{2d - (d-n) - t}{d - t} \quad \Longleftrightarrow \quad q < \tfrac{n}{d-t}. \end{displaymath}
	Since $\mu(P_{1}(\delta/10)) \sim_{C} |P_{1}|/|P|$, we may compute
	\begin{align*} 
	M \cdot \delta^{n(d-n)} \left(N\cdot \frac{\delta^{n-d}}{C|P|}\right)^q \cdot \frac{|P_{1}|}{|P|}  & \stackrel{\eqref{e:lowerbound}}{\lesssim_{C,d}} \int_{P_{1}(\delta/10)} \|\mu_x\|_{L^q(\mathcal{G}(d,n))}^q d\mu(x) \\
	& \stackrel{\eqref{form4}}{\leq} \int_{\mathcal{G}(d,n)} \|\pi_{V^{\perp}}\mu\|_{L^{q+1}(V^\perp)}^{q+1} d \gamma_{d,n}(V)  \stackrel{{\rm Theorem}\, \ref{t:LpProjections}}{\lesssim_{d,q,t}}  C_F 
	\end{align*}
for any $q < n/(d-t)$. Recall from \eqref{e:pigeon1} and \eqref{e:pigeon2} that $\tfrac{|\calI_{C\delta}(P, \calV)|}{|\calV_1|}\approx N$ and that $M \approx \tfrac{|\mathcal{I}_{C\delta}(P, \calV)|}{|P_1|}$. Hence, 
\begin{displaymath}
    	\frac{|\mathcal{I}_{C\delta}(P, \calV)|}{|P|} \cdot \delta^{n(d-n)} \cdot \left(\frac{|\calI_{C\delta}(P, \calV)|}{|\calV_1|} \cdot \frac{\delta^{n-d}}{C|P|}\right)^q \approx M\cdot \delta^{n(d-n)} \left(N\cdot \frac{\delta^{n-d}}{C|P|}\right)^q \frac{|P_{1}|}{|P|} \lesssim_{C,d,q,t} C_F
\end{displaymath}
for any $q < n/(d - t)$. If we now rearrange the equation above, and use the obvious inequalities $|\calV_1| \leq |\calV|$ and $C_{F}^{1/(q + 1)} \leq C_{F}$, we obtain
\begin{equation*}
		|\mathcal{I}_{C\delta}(P,\mathcal{V})| \lessapprox c(C,d,q,t) \cdot C_{F} \cdot |P| \cdot |\mathcal{V}|^{q/(q + 1)} \cdot \delta^{(q - n)(d - n)/(q + 1)}. 
\end{equation*}
Recall that ``$\lessapprox$'' hides a factor of the form $C_d\log(\delta^{-1})^{C_d}$ for some dimensional constant $C_d$. Choosing $q$ close enough to $n/(d - t)$, depending only on $\varepsilon$ and $C_d$, we have
\begin{equation*}
	C_d\log(\delta^{-1})^{C_d}\delta^{(q - n)(d - n)/(q + 1)} \lesssim_{d,\varepsilon,t} \delta^{n(t+1-d)(d-n)/(d+n-t) - \varepsilon},
\end{equation*}
Thus,
\begin{equation*}
	|\mathcal{I}_{C\delta}(P,\mathcal{V})| \lesssim_{C,d,\varepsilon,t} \delta^{-\varepsilon} \cdot C_{F} \cdot |P| \cdot |\mathcal{V}|^{q/(q + 1)} \cdot \delta^{n(t+1-d)(d-n)/(d+n-t)}. 
\end{equation*}
Finally, note that the factor $|\mathcal{V}|^{q/(q + 1)}$ is increasing in $q$, and so $|\mathcal{V}|^{q/(q + 1)}\le |\mathcal{V}|^{n/(d+n-t)}$. Together with the estimate above, this gives \eqref{main-inc-eq}.
\end{proof}

\section{Sharpness of the incidence estimate}\label{s:sharpness}

In this section we construct a family of examples showing that exponent in Theorem \ref{t:incidences} is sharp in the plane. More precisely, we consider the following family of problems, for each pair of parameters $s \in [0,1]$ and $t \in [1,2]$: let $P \subset [0,1]^{2}$ be a $(\delta,t,C)$-set with $t > 1$, and for some fixed constant $C > 1$. Assume that $\mathcal{L}_{s,t} \subset \mathcal{A}(2,1)$ is a $\delta$-separated family of lines with the property that every $p \in P$ is $\delta$-incident to at least $\delta^{-s}$ lines in $\mathcal{L}_{s,t}$: in other words the collections
\begin{displaymath} \mathcal{L}(p) := \mathcal{L}^{\delta}(p) := \{\ell \in \mathcal{L}_{s,t} : p \in \ell(\delta)\}, \qquad p \in P, \end{displaymath}
satisfy $|\mathcal{L}(p)| \geq \delta^{-s}$ for all $p \in P$. How many lines are there in $\mathcal{L}_{s,t}$? Theorem \ref{t:incidences} yields a lower bound, which (of course!) matches the numerology of Theorem \ref{mainFurstenberg}:
\begin{equation}\label{to1} |\mathcal{L}_{s,t}| \gtrsim_{C,\epsilon,t} \delta^{-\epsilon} \cdot \delta^{-2s - (1 - s)(t - 1)}. \end{equation}
This is not surprising, since Theorem \ref{mainFurstenberg} is proven by applying Theorem \ref{t:incidences}, see the next section. While it is highly unlikely that Theorem \ref{mainFurstenberg} is sharp, the lower bound \eqref{to1} is sharp for every $s \in [0,1]$ and $t \in [1,2]$:
\begin{proposition}\label{counterProp} For every $s \in [0,1]$ and $t \in [1,2]$, there exists 
\begin{enumerate}
\item\label{p1} a $\delta$-separated $(\delta,t)$-set $P \subset [0,1]^{2}$, and
\item\label{p2} a $c\delta$-separated set $\mathcal{L}_{s,t} \subset \mathcal{A}(2,1)$, where $c > 0$ is an absolute constant, such that 
\begin{displaymath} |\mathcal{L}_{s,t}| \lesssim \delta^{-2s - (1 - s)(t - 1)} \quad \text{and} \quad |\mathcal{L}(p)| \gtrsim \delta^{-s} \text{ for all } p \in P. \end{displaymath}
\end{enumerate}
\end{proposition}
All the implicit constants in Proposition \ref{counterProp} are absolute, and the $(\delta,t)$-set $P$ is, more precisely, a $(\delta,t,C)$-set for an absolute constant $C > 0$.

\begin{remark}\label{rem1} How can \eqref{to1} be sharp, while Theorem \ref{mainFurstenberg} is quite likely not? The reason is simple: in the context of Theorem \ref{mainFurstenberg}, the line family $\mathcal{L}_{s,t}$ has better separation properties than the family $\mathcal{L}_{s,t}$ in Proposition \ref{counterProp}. More precisely, Theorem \ref{mainFurstenberg} is roughly equivalent to the following discretised statement: \emph{if $P \subset [0,1]^{2}$ is a $\delta$-separated $(\delta,t)$-set, and every point $p \in P$ is $\delta$-incident to a $(\delta,s)$-set of lines $\mathcal{L}(p) \subset \mathcal{L}_{s,t}$, then $|\mathcal{L}_{s,t}| \gtrapprox \delta^{-2s - (1 - s)(t - 1)}$.} Now, the assumption that $\mathcal{L}(p)$ is a $(\delta,s)$-set implies that $|\mathcal{L}(p)| \gtrsim \delta^{-s}$ (as we also assume in Proposition \ref{counterProp}), but it contains more information on the separation of the lines in $\mathcal{L}(p)$. Proposition \ref{counterProp} shows that this information is needed to improve on the bound $2s + (1 - s)(t - 1)$ in Theorem \ref{mainFurstenberg}, for every $s \in (0,1)$ and $t \in [1,2)$. \end{remark}

We then begin the proof of Proposition \ref{counterProp}. For brevity of notation, we write
\begin{equation*}
	\eta = \eta(s,t) = {(1-s)(t-1)}, \qquad s \in [0,1], \, t \in [1,2].
\end{equation*}  
Consider $\tfrac{1}{2}\delta^{-\eta}$ horizontal tubes of width $\delta^{1-s}$ and length $1$, evenly distributed inside the unit cube (see Figure~\ref{fig:example}). We will denote the family of these tubes by $\calC$. Note that the sum of widths of tubes in $\calC$ is equal to 
\begin{equation*}
	\tfrac{1}{2} \cdot \delta^{1-s-\eta} = \tfrac{1}{2} \cdot \delta^{(2-t)(1-s)} \leq \tfrac{1}{2}.
\end{equation*}
Thus, the separation between the tubes is bounded from below by $|\calC|^{-1}/2 = \delta^{\eta}/2$. It it also worth pointing out that this separation is at least as large as the width $\delta^{1 - s}$ of the tubes (up to a constant), since $\delta^{\eta} = \delta^{(1 - s)(t - 1)} \geq \delta^{1 - s}$.
\begin{figure}
	\def\svgwidth{8cm}
\begingroup%
  \makeatletter%
  \providecommand\color[2][]{%
    \errmessage{(Inkscape) Color is used for the text in Inkscape, but the package 'color.sty' is not loaded}%
    \renewcommand\color[2][]{}%
  }%
  \providecommand\transparent[1]{%
    \errmessage{(Inkscape) Transparency is used (non-zero) for the text in Inkscape, but the package 'transparent.sty' is not loaded}%
    \renewcommand\transparent[1]{}%
  }%
  \providecommand\rotatebox[2]{#2}%
  \newcommand*\fsize{\dimexpr\f@size pt\relax}%
  \newcommand*\lineheight[1]{\fontsize{\fsize}{#1\fsize}\selectfont}%
  \ifx\svgwidth\undefined%
    \setlength{\unitlength}{680.85952247bp}%
    \ifx\svgscale\undefined%
      \relax%
    \else%
      \setlength{\unitlength}{\unitlength * \real{\svgscale}}%
    \fi%
  \else%
    \setlength{\unitlength}{\svgwidth}%
  \fi%
  \global\let\svgwidth\undefined%
  \global\let\svgscale\undefined%
  \makeatother%
  \begin{picture}(1,0.72530914)%
    \lineheight{1}%
    \setlength\tabcolsep{0pt}%
    \put(0.75965023,0.22166152){\color[rgb]{0,0,0}\makebox(0,0)[lt]{\lineheight{1.25}\smash{\begin{tabular}[t]{l}$\delta^{1-s}$\end{tabular}}}}%
    \put(0,0){\includegraphics[width=\unitlength,page=1]{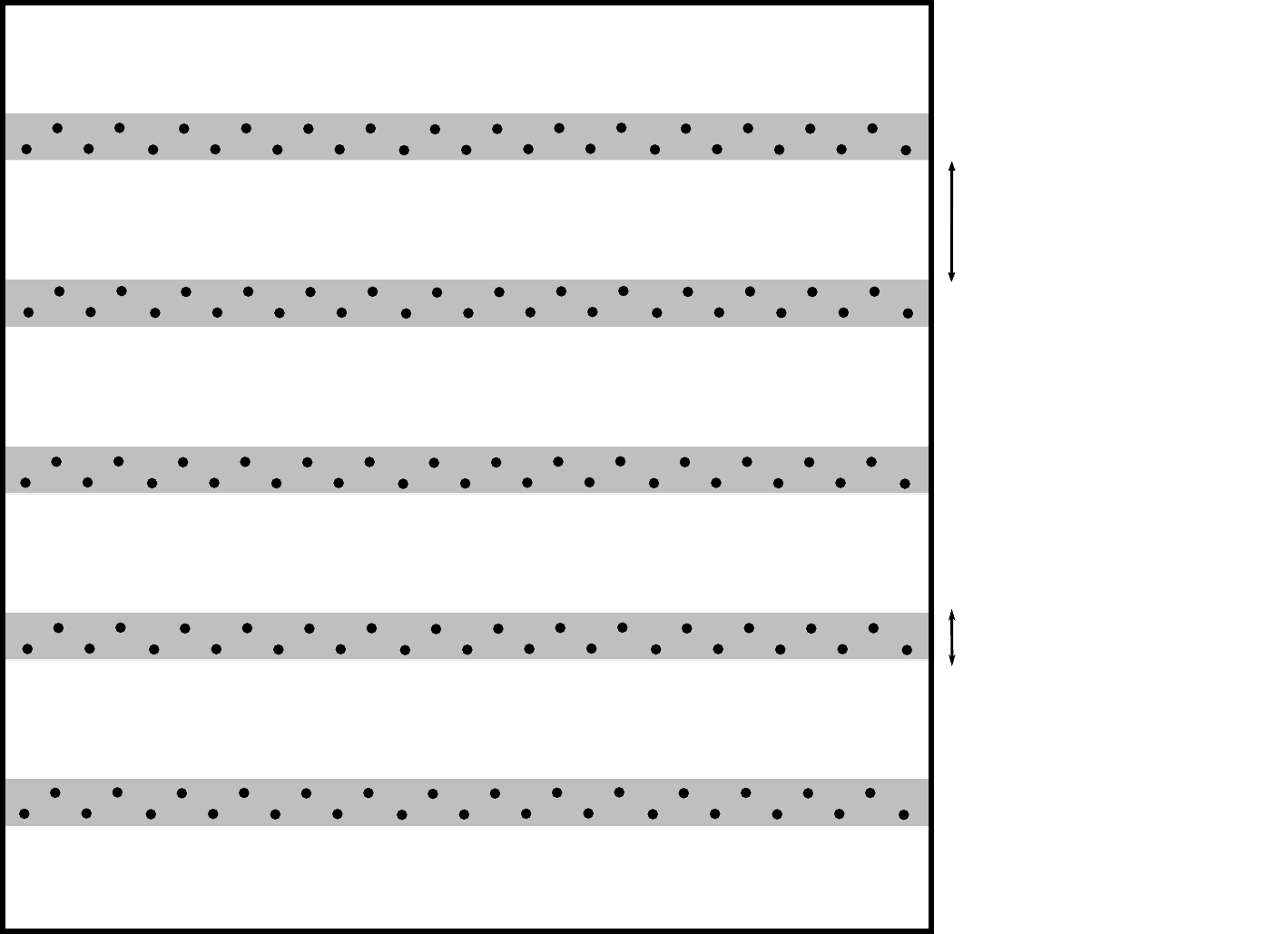}}%
    \put(0.76068872,0.54248318){\color[rgb]{0,0,0}\makebox(0,0)[lt]{\lineheight{1.25}\smash{\begin{tabular}[t]{l}$\ge{\delta^{\eta}}/{2}$\end{tabular}}}}%
  \end{picture}%
\endgroup%

	\caption{The construction in Proposition \ref{counterProp}.}\label{fig:example}
\end{figure}

Inside each $C\in\calC$ we place $\sim \delta^{-t+\eta}$ points, distributed uniformly, see Figure \ref{fig:example}. We denote the sets so obtained $P_C$, $C \in \mathcal{C}$. With this definition, the points in $P_{C}$ are (at least) $\delta$-separated, since 
\begin{equation*}
	|P_C|\,\delta^{2} = \delta^{-t+\eta+2}\le \calH^2(C) = \delta^{1-s},
\end{equation*}
where the inequality follows from the fact that $-t+\eta+1+s\ge 0$.

Setting $P := \bigcup_{C\in\calC} P_C$, we see that $|P| \sim \delta^{-t+\eta} \cdot |\calC| \sim \delta^{-t}$. This was just a preliminary observation to convince the reader that $P$ might be a $(\delta,t)$-set, as we will prove a little later. One useful property of $P_{C}$, $C \in \mathcal{C}$, is that given a ball $B$ with radius $\delta\le r \le 1$ we have
\begin{equation}\label{eq:PCcapB}
	|P_C\cap B| \lesssim \frac{\calH^2(C\cap B)}{\calH^2(C)}\,  |P_C| + 1\sim \delta^{-t+\eta -1 +s}\,\calH^2(C\cap B) + 1.
\end{equation}

Before proving that $P$ is a $(\delta,t)$-set, we define the family of lines $\mathcal{L}_{s,t}$, and verify the properties stated in Proposition \ref{counterProp}\eqref{p2}. First, we define an appropriate set of directions $\Sigma\subset S^1$. Let $e_1=(1,0)\in S^1$ and let $\Sigma \subset B(e_{1},\delta^{1 - s})\subset S^1$ be a $\delta$-net, so that $|\Sigma| \sim \delta^{-s}$. For every thick horizontal tube $C\in\calC$ we define $\mathcal{L}_C$ to be a $c\delta$-net among those lines in $\mathcal{A}(2,1)$ which have directions in $\Sigma$ and which intersect $C$. It follows from elementary geometry that for each fixed direction $e\in \Sigma$ there are $\sim \delta^{-s}$ lines in $\mathcal{L}_{C}$ with direction $e$ (see Figure \ref{fig:covering}). Hence,
\begin{equation*}
	|\calL_C| \lesssim \delta^{-s}|\Sigma| \sim \delta^{-2s}.
\end{equation*}
We then set
\begin{equation*}
	\calL_{s,t} =\bigcup_{C\in\calC}\calL_C,
\end{equation*}
so that
\begin{equation*}
	|\calL_{s,t}| \le |\calC|\cdot|\calL_C| \lesssim \delta^{-2s-\eta},
\end{equation*}
as claimed in Proposition \ref{counterProp}\eqref{p2}.
\begin{figure}
	\def\svgwidth{6cm}
\begingroup%
  \makeatletter%
  \providecommand\color[2][]{%
    \errmessage{(Inkscape) Color is used for the text in Inkscape, but the package 'color.sty' is not loaded}%
    \renewcommand\color[2][]{}%
  }%
  \providecommand\transparent[1]{%
    \errmessage{(Inkscape) Transparency is used (non-zero) for the text in Inkscape, but the package 'transparent.sty' is not loaded}%
    \renewcommand\transparent[1]{}%
  }%
  \providecommand\rotatebox[2]{#2}%
  \newcommand*\fsize{\dimexpr\f@size pt\relax}%
  \newcommand*\lineheight[1]{\fontsize{\fsize}{#1\fsize}\selectfont}%
  \ifx\svgwidth\undefined%
    \setlength{\unitlength}{773.06377407bp}%
    \ifx\svgscale\undefined%
      \relax%
    \else%
      \setlength{\unitlength}{\unitlength * \real{\svgscale}}%
    \fi%
  \else%
    \setlength{\unitlength}{\svgwidth}%
  \fi%
  \global\let\svgwidth\undefined%
  \global\let\svgscale\undefined%
  \makeatother%
  \begin{picture}(1,0.68934932)%
    \lineheight{1}%
    \setlength\tabcolsep{0pt}%
    \put(0,0){\includegraphics[width=\unitlength,page=1]{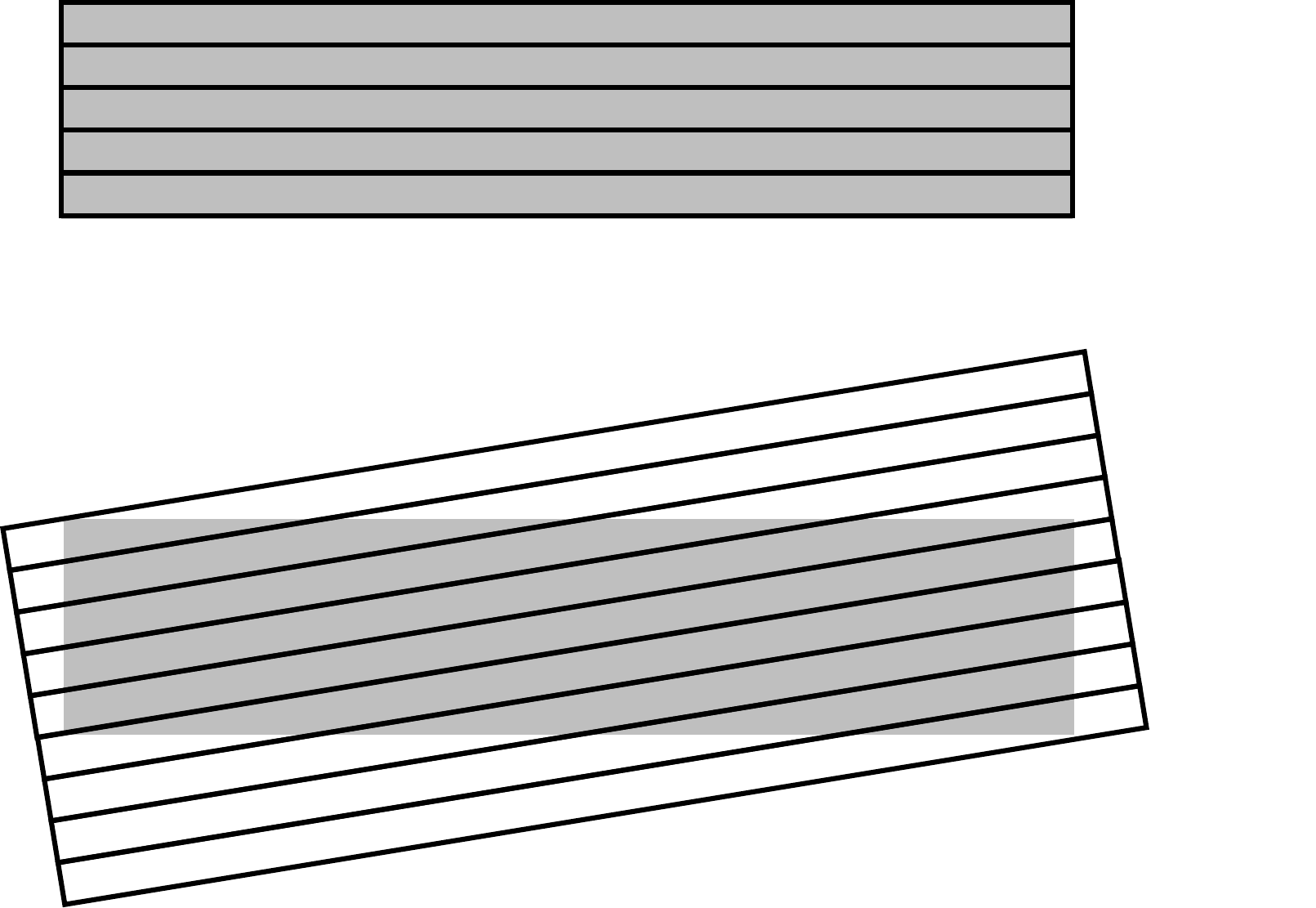}}%
    \put(0.87825946,0.57659496){\color[rgb]{0,0,0}\makebox(0,0)[lt]{\lineheight{1.25}\smash{\begin{tabular}[t]{l}$\delta^{1-s}$\end{tabular}}}}%
    \put(0,0){\includegraphics[width=\unitlength,page=2]{covering-tubes.pdf}}%
  \end{picture}%
\endgroup%

	\caption{In the definition of $\calL_C$, we choose for every $e \in \Sigma \subset B(e_{1},\delta^{1 - s})$ a $\delta$-net of lines $\mathcal{L}_{C}$ intersecting $C$, with direction $e$. For $e\in \Sigma$ fixed, there are $\sim \delta^{-s}$ lines in $\mathcal{L}_{C}$ with direction $e$. This is trivial if $e = e_{1}$ (first picture), and takes some easy trigonometry for general $e \in \Sigma$ (second picture). }\label{fig:covering}
\end{figure}

Observe that for every fixed $e\in \Sigma$ and $p \in P_{C}$, some line in $\mathcal{L}_{C}$ with direction $e$ is $\delta$-incident to $p$. Therefore, $|\mathcal{L}(p)| \gtrsim \delta^{-s}$ for every $p \in P$, as claimed in Proposition \ref{counterProp}\eqref{p2}.

To complete the proof of Proposition \ref{counterProp}, it remains to verify that $P$ is a $(\delta,t)$-set. 

\begin{lemma}
	For any ball $B$ with radius $\delta^{\alpha},\ 0\le \alpha \le 1,$ we have
	\begin{equation}\label{eq:tsetcond}
		|P\cap B| \lesssim \delta^{\alpha t-t} \sim \delta^{\alpha t}\, |P|.
	\end{equation}
\end{lemma}
\begin{proof}
	Let $0\le \alpha\le 1$, and let $B$ be a ball of radius $r(B)=\delta^{\alpha}$ that intersects $P$. There are three cases to consider.
	\vspace{0.5em}
	
	\subsubsection*{Case $1 - s < \alpha \leq 1$} Note that the radius of $B$ is smaller than the width of the tubes in $\calC$, so $B$ intersects at most $3$ tubes from $\calC$. Let $C\in\calC$ be one of these tubes. Note that $\calH^2(C\cap B)\lesssim \delta^{2\alpha}$, and consequently
	\begin{equation*}
		|P_C\cap B| \overset{\eqref{eq:PCcapB}}{\lesssim} \delta^{-t+\eta -1 +s}\,\calH^2(C\cap B)+1 \lesssim \delta^{2\alpha-t+\eta -1 +s}+1.
	\end{equation*}
	We need to check if the right hand side is bounded by $\delta^{\alpha t - t}$. The bound $1\le \delta^{\alpha t - t}$ is trivial, since $\alpha\le 1$. So we only need to bound $\delta^{2\alpha-t+\eta -1 +s}$. This amounts to verifying that
	\begin{equation*}
		2\alpha + \eta -1+s -\alpha t\ge 0 \quad \Longleftrightarrow \quad (1-s-\alpha)(t-2) \ge 0.
	\end{equation*}
	This is true because we assume $\alpha\ge 1-s$ and $t \leq 2$. This shows \eqref{eq:tsetcond} for $1-s < \alpha \leq 1$.
	\vspace{0.5em}
	
\subsubsection*{Case $\eta \leq \alpha \leq 1 - s$} Note that $\eta=(1-s)(t-1) \leq 1-s$, so $[\eta,1 - s] \neq \emptyset$. Recall that the separation between the tubes in $\mathcal{C}$ was at least $\delta^\eta/2$. Since $r(B) \leq \delta^\eta$, it follows that $B$ intersects at most $3$ tubes from $\calC$. Let $C\in\calC$ be one of these tubes. Observe that, since the radius of $B$ is larger than the width of $C$, we have
	\begin{equation*}
		\calH^2(C\cap B)\lesssim \delta^{\alpha}\calH^2(C) = \delta^{\alpha+1-s}.
	\end{equation*}
	Hence,
	\begin{equation*}
		|P_C\cap B| \overset{\eqref{eq:PCcapB}}{\lesssim} \delta^{-t+\eta -1 +s}\,\calH^2(C\cap B) + 1\\
		\lesssim \delta^{-t+\eta +\alpha} + 1.
	\end{equation*}
	It is, again, clear that $1\le \delta^{\alpha t - t}$. So we only need to check that
	\begin{equation*}
		\delta^{-t+\eta+\alpha}\le \delta^{\alpha t-t} \quad \Longleftrightarrow \quad \eta + \alpha - \alpha t\ge 0 \quad \Longleftrightarrow \quad (1-s-\alpha)(t-1)\ge 0.
	\end{equation*}
	This is true because $t \geq1$ and $1-s\ge \alpha$.
	\vspace{0.5em}
	
\subsubsection*{Case $0 \leq \alpha \leq \eta$} Note that, in particular, $\alpha \le 1-s$ holds in this case. Observe that since the tubes in $\mathcal{C}$ are $(\delta^\eta/2)$-separated, $B$ intersects $\lesssim \delta^{\alpha-\eta}$ tubes in $\mathcal{C}$.
	
	As in the previous case, for every tube $C \in \mathcal{C}$ we have $\calH^2(C\cap B)\lesssim \delta^{\alpha+1-s}$. Thus,
	\begin{multline}\label{eq:PcapB}
		|P\cap B| =\sum_{C\in\calC}|P_C\cap B| \overset{\eqref{eq:PCcapB}}{\lesssim} \sum_{C\in\calC} \delta^{-t+\eta -1 +s}\,\calH^2(C\cap B) + |\{C\in\calC\, :\, C\cap B\neq\emptyset\} |\\
		\lesssim \delta^{\alpha-\eta}\,\delta^{-t+\eta+\alpha} + \delta^{\alpha-\eta} = \delta^{2\alpha-t} + \delta^{\alpha-\eta}.
	\end{multline}
	Clearly $\delta^{2\alpha - t} \leq \delta^{\alpha t - t}$, since $t \leq 2$. It remains to show that $\delta^{\alpha-\eta}\le \delta^{\alpha t-t}$. In fact, it even turns out that $\delta^{\alpha - \eta} \leq \delta^{2\alpha - t}$, or equivalently $\alpha + \eta \le t$. Since $\alpha\le\eta$ in the current case, we have $\alpha + \eta \leq 2\eta$, so it suffices to show that $2\eta \leq t$. Recalling once more that $\eta=(1-s)(t-1)$, this is equivalent to
	\begin{equation*}
		(2-t) + 2s(t-1)\ge 0.
	\end{equation*}
	This is true for every $s\in [0,1]$ and $t\in [1,2]$. This completes the proof of \eqref{eq:tsetcond}, and hence that of Proposition \ref{counterProp}.
\end{proof}

\section{Application to Furstenberg sets}\label{s:furstenberg}
In this section we prove \thmref{mainFurstenberg}, which states that every $(d - 1,s,t)$-Furstenberg set $K \subset \R^{d}$, with $1 < t \leq d$ and $0 < s \leq d - 1$ satisfies
\begin{equation} \Hd K \geq (2s + 2 - d) + \frac{(t - 1)(d - 1 - s)}{d - 1}. \end{equation}
First, we define $\delta$-discretised Furstenberg sets.
\begin{definition}
	We say that $F\subset B(2)\subset\R^d$ is a \emph{$\delta$-discretised $(n,s,t)$-Furstenberg set} if
	\begin{itemize}
		\item there exists a $\delta$-separated $(\delta,t)$-set of $n$-planes $\mathcal{V}\subset\calA(d,n)$,
		\item $F=\bigcup_{V\in\calV} F_V$, where each $F_V$ is a union of $\delta$-balls,
		\item $F_V$ is a $(\delta,s)$-set contained in $V(2\delta)$.
	\end{itemize}
\end{definition}

We will use the following lemma due to H\'{e}ra, Shmerkin, and Yavicoli \cite[Lemma 3.3]{2020arXiv200111304H}.

\begin{lemma}\label{lem:discretisedFurst}
	Suppose that every $\delta$-discretised $(n,s,t)$-Furstenberg set, $\delta \in (0,1]$, has Lebesgue measure $\gtrsim \delta^{d-\alpha}$. Then every $(n,s,t)$-Furstenberg set has Hausdorff dimension at least $\alpha$.
\end{lemma}
The lemma above was proved in \cite{2020arXiv200111304H} only for $n=1$, but the proof for $1<n<d$ is virtually the same. Now, to prove \thmref{mainFurstenberg} it suffices to show that every $\delta$-discretised $(d-1,s,t)$-Furstenberg set $F$, with $1 < t \leq d$ and $0 < s \leq d - 1$, satisfies
\begin{equation*}
\calH^d(F)\gtrsim \delta^{d-\alpha}
\end{equation*}
for any $\alpha<\alpha_0:=(2s + 2 - d) + \frac{(t - 1)(d - 1 - s)}{d - 1}$. Actually, we will prove a slightly stronger result.

\begin{proposition}\label{prop:Furst}
	Assume that $t \in (1,d]$, $s \in (0,d - 1]$, and $\mathfrak{c} > 0$. Let $\mathcal{V}\subset\calA(d,d-1)$ be a $\delta$-separated $(\delta,t)$-set, with $\delta \in (0,1]$. For each $V\in\calV$ let $F_V\subset V(2\delta)\cap B(2)$ be a union of at least $\mathfrak{c}\delta^{-s}$ disjoint $\delta$-balls. If $F=\bigcup_{V\in\calV} F_V$, then for any $\alpha<\alpha_0$
	\begin{equation}\label{eq:propFurst}
	\calH^d\del{F}\gtrsim \delta^{d-\alpha},
	\end{equation}
	with implicit constant depending on $\alpha,\mathfrak{c},d,t$.
\end{proposition}
Note that compared to the definition of $\delta$-discretised $(d-1,s,t)$-Furstenberg sets, we do not need to assume that $F_V$ is a $(\delta,s)$-set; the cardinality estimate for the number of $\delta$-balls is sufficient. Of course, every $\delta$-discretised $(d-1,s,t)$-Furstenberg set satisfies the assumptions of \propref{prop:Furst} because our definition of $(\delta,s)$-sets implies the desired cardinality lower bound.

The proof of \propref{prop:Furst} can be summarized as follows: use point-plane duality and apply \thmref{t:incidences}. We provide the details below.

\subsection{Duality}
	Consider a map $\textbf{D}:\R^d\to\calA(d,d-1)$ given by
	\begin{equation*}
	(x_1,\dots,x_d)\mapsto\set{(y_1,\dots,y_{d-1}, \sum_{i=1}^{d-1}x_iy_i + x_d)\ :\ (y_1,\dots,y_{d-1})\in\R^{d-1}}.
	\end{equation*}
	The image of $\DD$ consists of all the $(d-1)$-planes that do not contain a translate of the vertical line $\set{(0,\dots,0,y_d)\ :\ y_d\in\R}$, or equivalently, the $(d-1)$-planes whose orthogonal projection to the horizontal plane $\DD(0) = \R^{d - 1} \times \{0\}$ is the whole plane.
	
	A direct computation shows that
	\begin{multline*}
	d_{\calA}(\DD(x),\DD(y)) = 
	\abs{\frac{(x_1,\dots,x_{d-1},-1)}{\abs{(x_1,\dots,x_{d-1},-1)}} -\, \frac{(y_1,\dots,y_{d-1},-1)}{\abs{(y_1,\dots,y_{d-1},-1)}}}\\
	 +\ \abs{\frac{x_d}{\abs{(x_1,\dots,x_{d-1},-1)}} -\, \frac{y_d}{\abs{(y_1,\dots,y_{d-1},-1)}}}.
	\end{multline*}	
	Hence, for any given $0<R<\infty$ the restriction of $\DD$ to $B(R)$ is bilipschitz onto its image, with bilipschitz constant depending only on $R$ and $d$. In particular, $\DD$ is injective. Write $\DD(0) = \R^{d - 1} \times \{0\} =: V_0$, and observe that there exists a dimensional constant $0<r_d<1$ such that $B(V_0,r_d)\subset \DD(B(1))$.
	
	Consider now the map $\DD^*:\im\DD \to \R^d$ defined by
	\begin{equation*}
	V = \DD(x_1,\dots, x_d) \mapsto (-x_1,\dots,-x_{d-1},x_d).
	\end{equation*}
	In other words, $\DD^*$ is the inverse of $\DD$ composed with reflection over the vertical line. The map $\DD^*$ was defined this way in order to preserve the incidence relation: for $x\in\R^d$ and $V\in\im\DD$, it holds
	\begin{equation}\label{eq:dualrelation}
	x\in V\quad\Longleftrightarrow\quad\DD^*(V)\in\DD(x).
	\end{equation}
	Indeed, $x \in V = \mathbf{D}(y_{1},\ldots,y_{d})$ is equivalent to $x_{d} = \sum_{i = 1}^{d - 1} x_{i}y_{i} + y_{d}$, which is equivalent to $y_{d} = \sum_{i = 1}^{d - 1} (-y_{i})x_{i} + x_{d}$, which is equivalent to $\mathbf{D}^{\ast}(V) = (-y_{1},\ldots,-y_{d - 1},y_{d}) \in \mathbf{D}(x)$.
	Note that the restriction of $\DD^*$ to $B(V_0,r_d)\subset \DD(B(1))$ is bilipschitz onto its image, by our earlier remarks, and that $\DD^*(B(V_0,r_d))\subset\DD^*(\DD(B(1)))=B(1)$. We will also need the following quantitative version of \eqref{eq:dualrelation}.
		
	\begin{lemma}
		For $x\in B(2)$ and $V\in B(V_0,r_d)$ we have
		\begin{equation}\label{eq:dualrelationquant}
		\frac{\dist(\DD^*(V),\DD(x))}{3}\le\dist(x,V)\le 3\dist(\DD^*(V),\DD(x)).
		\end{equation}
	\end{lemma}
	\begin{proof}
		Let $p=(p_1,\dots,p_d)\in B(1)$ be the unique point such that $V=\DD(p)$. A direct computation yields
		\begin{equation*}
		\dist(x,\DD(p)) = \frac{\abs{p_d - x_d + \sum_{i=1}^{d-1}x_ip_i}}{|(p_1,\dots,p_{d-1},-1)|},
		\end{equation*}
		and
		\begin{equation*}
		\dist(\DD^*(V),\DD(x)) = \frac{\abs{p_d-x_d + \sum_{i=1}^{d-1}x_ip_i}}{|(x_1,\dots,x_{d-1},-1)|}.
		\end{equation*}
		Since $1\le |(p_1,\dots,p_{d-1},-1)|\le 2$ and $1\le |(x_1,\dots,x_{d-1},-1)|\le 3$, \eqref{eq:dualrelationquant} follows.
	\end{proof}
	
	Let $F\subset B(2)$ and $\calV \subset \mathcal{A}(d,d - 1)$ be as in \propref{prop:Furst}, and let $P\subset F$ be a maximal $\delta$-separated subset of $F$. Evidently each plane $V \in \mathcal{V}$ intersects $B(3)$, so $\mathcal{V} \subset B(V_{0},7)$. After this observation, a few standard steps allow us to reduce to the case $\mathcal{V} \subset B(V_{0},r_{d}) \subset \mathbf{D}(B(1))$. In particular $\mathbf{D}^{\ast}$ is $C_{d}$-bilipschitz on $\mathcal{V}$.

	We now define
	\begin{equation}\label{to2}
	\calV_D := \DD(P) := \{\DD(p) : p \in P\} \subset \mathcal{A}(d,d - 1) \quad \text{and} \quad P_{D} := \DD^{\ast}(\mathcal{V}) \subset B(1).
	\end{equation}
	Observe that since $P$ is $\delta$-separated, and $P \subset B(2)$, the collection $\calV_D$ is $c\delta$-separated for some $c = c_{d} > 0$, by the local bilipschitz property of $\DD$. Also, since $\mathcal{V}$ was assumed to be a $\delta$-separated $(\delta,t)$-set, $P_{D} \subset B(1)$ is a $c\delta$-separated $(\delta,t)$-set (with explicit and implicit constants depending on "$d$" only).
	
	\subsection{Applying the incidence bound}
	We wish to apply \thmref{t:incidences} with $\calV_D$ and $P_D$ as above. Recall that
		\begin{itemize}
		\item $\calV_D$ is $c\delta$ separated,
		\item $P_D\subset B(1)$ is a $c\delta$-separated $(\delta,t)$-set.
	\end{itemize}
	Moreover, by \eqref{eq:dualrelationquant} and the assumptions on $\calV$ and $F$, for each $p \in P_D$ there exists a $c\delta$-separated set $\calV_D(p) \subset \calV_D$ such that $|\calV_D(p)| \ge \mathfrak{c}\delta^{-s}$, and for each $V\in\calV_D(p)$ we have $\dist(p,V)\le 6\delta = (6/c) \cdot c\delta$. This numerology places us in a position to apply \thmref{t:incidences} at scale $\delta' := c\delta$, with "thickening" constant $C := 6/c \sim_{d} 1$. To simplify notation, we omit the apostrophe, and write "$\delta$" in place of "$\delta'$".
	
\begin{proof}[Proof of \propref{prop:Furst}]
		Applying \thmref{t:incidences} to $\calV_D,\ P_D,$ and some small $\varepsilon>0$, we arrive at
		\begin{equation*}
		|\mathcal{I}_{C\delta}(P_D,\mathcal{V}_D)| \lesssim_{d,\varepsilon,t} \delta^{-\varepsilon}\cdot |P_D| \cdot |\mathcal{V}_D|^{(d-1)/(2d-t-1)} \cdot \delta^{(d-1)(t+1-d)/(2d-t-1)}.
		\end{equation*}
		Noting that each $p\in P_D$ is $C\delta$-incident to the $\geq \mathfrak{c}\delta^{-s}$ planes $\mathcal{V}_{D}(p) \subset \calV_D$, we get that
		\begin{equation*}
		\mathfrak{c}\delta^{-s}|P_D| \lesssim_{d,\varepsilon,t} \delta^{-\varepsilon}\cdot |P_D| \cdot |\mathcal{V}_D|^{(d-1)/(2d-t-1)} \cdot \delta^{(d-1)(t+1-d)/(2d-t-1)}.
		\end{equation*}
		Setting $\varepsilon_0:=\varepsilon(2d-t-1)/(d-1)$ we arrive at
		\begin{equation*}
		|\calV_D|\gtrsim_{\mathfrak{c},d,\varepsilon,t} \delta^{- t-1+d-s(2d-t-1)/(d-1) + \varepsilon_0}.
		\end{equation*}
		Recall from \eqref{to2} that $|P|\ge|\calV_D|$, where $P$ is a maximal $\delta$-separated subset of $F$, and $F$ is a union of $\delta$-balls. Hence,
		\begin{equation*}
		\calH^d\del{F}\gtrsim |P| \cdot \delta^d \gtrsim_{\mathfrak{c},d,\varepsilon,t} \delta^{d- t-1+d-s(2d-t-1)/(d-1) + \varepsilon_0}.
		\end{equation*}
		A simple computation shows that
		\begin{equation*}
		t+1-d+\frac{s(2d-t-1)}{d-1}=(2s + 2 - d) + \frac{(t - 1)(d - 1 - s)}{d - 1}=\alpha_0,
		\end{equation*}
		and since we may choose $\varepsilon$ arbitratrily small, we get \eqref{eq:propFurst}.
	\end{proof}

\subsection{Application to the sum-product problem}\label{s:sumProduct} In this short section, we derive Corollary \ref{cor:sumProduct} from Proposition \ref{prop:Furst}. Recall that Corollary \ref{cor:sumProduct} claims the following: if $A \subset [1,2]$ is a $\delta$-separated set with $|A| = \delta^{-s}$, $B\subset[1,2]$ is a $\delta$-separated  $(\delta,t,c)$-set, $C\subset[1,2]$ is a $\delta$-separated  $(\delta,t',c')$-set, and $t+t'>1$, then for any $\varepsilon>0$
\begin{equation}\label{form41} \max\{|A + B|_{\delta},|A \cdot C|_{\delta}\} \gtrsim_{\varepsilon,s,t,t'c, c'} \delta^{-(t + t' - 1)(1 - s)/2+\varepsilon}|A|, \end{equation} 
Given Proposition \ref{prop:Furst}, this follows from a well-known argument of Elekes \cite{MR1472816}, repeated below. Consider the $\delta$-neighbourhood 
\begin{displaymath} F := [(A + B) \times (A \cdot C)](\delta) \subset \R^{2}. \end{displaymath}
Consider also the family of planar lines
\begin{equation*}
\mathcal{L} := \{y = cx - bc : b \in B,\, c\in C\}.
\end{equation*}
Thus $\mathcal{L}$ contains $|B|$ lines for every fixed slope $c \in C$, and in total $|\mathcal{L}| = |B|\cdot |C|$. It is not hard to check that $\mathcal{L}$ is a $c_0\delta$-separated $(\delta,t+t',c_1)$-set of lines, where $c_0 > 0$ is absolute, and $c_1 > 0$ only depends on $c, c'$. To give a few more details, if $(a,b) \mapsto \mathbf{D}(a,b) := \{y = ax + b : x \in \R\}$ is the duality map $\R^{2} \to \mathcal{A}(2,1)$, then $\mathcal{L} = \mathbf{D}(\{(c,-bc) : b \in B,\, c\in C\})$. Here $\{(c,-bc) : b \in B,\, c\in C\} \subset \R^{2}$ is a $(\delta,t+t',c_1')$-set, since it is the image of the $(\delta,t+t',c_1'')$-set $C \times B \subset [1,2]^{2}$ under $(x,y) \mapsto R(x,y) = (x,-xy)$, which is bilipschitz on $[1,2]^{2}$.

Now observe that if $\ell = \{(x,cx - bc) : x \in \R\} \in \mathcal{L}$, then $\ell$ contains the set
\begin{displaymath} F_{\ell} := \{(a + b,ac) : a \in A\} \subset (A + B) \times (A \cdot C) \subset F. \end{displaymath}
The set $F_{\ell}$ is an affine copy of $A$, and it is easy to see that it is $\delta$-separated and satisfies $|F_{\ell}| = |A| = \delta^{-s}$, for every $\ell \in \mathcal{L}$. Since $F$ contains the union of (the $\delta$-neighbourhoods of) the sets $F_{\ell}$ for $\ell \in \mathcal{L}$, it follows from Proposition \ref{prop:Furst} that
\begin{displaymath} \delta^{2} \cdot |A + B|_{\delta} \cdot |A \cdot C|_{\delta} \sim \mathcal{L}^{2}(F) \gtrsim_{\alpha,s,t,t',c, c'} \delta^{2 - \alpha}, \qquad \alpha < 2s + (t +t' - 1)(1 - s). \end{displaymath}
This yields \eqref{form41}, and therefore Corollary \ref{cor:sumProduct}.	

\bibliographystyle{plainurl}
\bibliography{references}

\end{document}